\UseRawInputEncoding 
\documentclass[10pt]{amsart}
\pdfoutput=1
\usepackage{amsmath, amsthm, amssymb,slashed}
\usepackage{ifpdf}
\usepackage[pdftex]{graphicx}
\usepackage{tikz}
\usetikzlibrary{matrix,arrows,calc}
\usepackage[pdftex,plainpages=false,hypertexnames=false,pdfpagelabels]{hyperref}
 \theoremstyle{plain}
 \newtheorem{thm}{Theorem}[section]
 
 \newtheorem{lem}[thm]{Lemma}
  
 \newtheorem{prop}[thm]{Proposition}
 \theoremstyle{definition}
 \newtheorem{defn}[thm]{Definition}

 \theoremstyle{remark}
 \newtheorem{rmk}[thm]{Remark}

\def\beq{\begin{eqnarray}}
\def\eeq{\end{eqnarray}}

\DeclareSymbolFont{bbold}{U}{bbold}{m}{n}
\DeclareSymbolFontAlphabet{\mathbbold}{bbold}

 \newcommand{\bp}{\begin{proof}[Proof]}
 \newcommand{\ep}{\end{proof}}

\DeclareMathOperator{\Pf}{{\rm Pf}}
\DeclareMathOperator{\Euc}{{\sf Euc}}

\newcommand{\sq}{\mathord{/\!\!/}}

\DeclareMathOperator{\dR}{\rm dR} 

\def\Eu{{\rm Eu}}
\def\dil{{u}}
\def\sTr{{\rm sTr}}
\def\dR{{\rm d}}
\def\red{{\rm{red}}}
\def\Map{{\sf{Map}}}
\def\HH{{\mathbb{H}}}
\def\Ell{{\mathcal{E}\ell\ell}}

\def\pt{{\rm{pt}}}

\def\Lat{{\sf Lat}}

\def\F{\mathcal{F}}
\def\odd{{\rm odd}}
\def\ev{{\rm ev}}
\def\cl{{\rm cl}}

\def\vol{{\rm vol}}

\def\R{{\mathbb{R}}}
\def\A{{\mathbb{A}}}

\def\id{{{\rm id}}}

\def\MF{{\rm MF}}

\def\HH{{\mathbb{H}}}
\def\K{{\rm {K}}}

\def\H{{\rm {H}}}
\def\C{{\mathbb{C}}}

\def\Z{{\mathbb{Z}}}
\def\E{{\mathbb{E}}} 

\def\Fun{{\sf Fun}}

\def\End{\mathop{\sf End}}

\def\Bord{\hbox{-{\sf Bord}}}

\def\EFT{ \hbox{-{\sf EFT}}}

\def\Spin{{\rm Spin}}

\def\SL{{\rm SL}}

\def\TMF{{\rm TMF}}
\def\EBord{ \hbox{-{\sf EBord}}}

\newcommand{\op}{{\sf{op}}}

\begin{document}

\title{Chern characters for supersymmetric field theories}

\author{Daniel Berwick-Evans}

\address{Department of Mathematics, University of Illinois at Urbana-Champaign}

\email{danbe@illinois.edu}


\begin{abstract} 
We construct a map from $d|1$-dimensional Euclidean field theories to complexified K-theory when $d=1$ and complex analytic elliptic cohomology when $d=2$. This provides further evidence for the Stolz--Teichner program, while also identifying candidate geometric models for Chern characters within their framework. The construction arises as a higher-dimensional and parameterized generalization of Fei Han's realization of the Chern character in K-theory as dimensional reduction for $1|1$-dimensional Euclidean field theories. In the elliptic case, the main new feature is a subtle interplay between the geometry of the super moduli space of $2|1$-dimensional tori and the derived geometry of complex analytic elliptic cohomology. As a corollary, we obtain an entirely geometric proof that partition functions of $\mathcal{N}=(0,1)$ supersymmetric quantum field theories are weak modular forms, following a suggestion of Stolz and Teichner. 
\end{abstract}

\maketitle 

\section{Introduction and statement of results}

Given a smooth manifold $M$, Stolz and Teichner have constructed categories of $d|1$-dimensional super Euclidean field theories over~$M$ for $d=1,2$
\beq
d|1\EFT(M):=\Fun^{\otimes}(d|1\EBord(M),\mathcal{V})\label{eq:defnFT}
\eeq
whose objects are symmetric monoidal functors from a bordism category $d|1\EBord(M)$ to a category of vector spaces $\mathcal{V}$. The morphisms of $d|1\EBord(M)$ are $d|1$-dimensional super Euclidean bordisms with a map to a smooth manifold~$M$. We refer to~\cite[\S4]{ST11} for details. Stolz and Teichner have conjectured the existence of cocycle maps~\cite[\S1.5-1.6]{ST11}
\beq
&&\begin{tikzpicture}[baseline=(basepoint)];
\node (A) at (0,0) {$1|1\EFT(M)$};
\node (B) at (3,0) {$\K(M),$};
\node (C) at (6,0) {$2|1\EFT(M)$};
\node (D) at (9,0) {$\TMF(M),$};
\draw[->>,dashed] (A) to node [above] {\small cocycle}  (B);
\draw[->>,dashed] (C) to node [above] {\small cocycle} (D);
\path (0,-.05) coordinate (basepoint);
\end{tikzpicture}\label{eq:conjecture}
\eeq
for $\K$-theory and the cohomology theory of topological modular forms ($\TMF$). In this paper we construct subcategories $\mathcal{L}^{d|1}_0(M)\subset d|1\EBord(M)$ consisting of super circles with maps to~$M$ when $d=1$ and super tori with maps to~$M$ when $d=2$, both viewed as a particular class of closed bordisms over~$M$. A super Lie group $\Euc_{d|1}$ acts through super Euclidean isometries on super circles and super tori, inducing actions on~$\mathcal{L}^{d|1}_0(M)$ for $d=1,2$.

\begin{thm} \label{thm}
Invariant functions $C^\infty(\mathcal{L}^{d|1}_0(M))^{\Euc_{d|1}}$ determine cocycles in  2-periodic cohomology with complex coefficients when $d=1$ and cohomology with coefficients in the ring $\MF$ of weak modular forms when $d=2$. Composing with restriction along $\mathcal{L}^{d|1}_0(M)\subset d|1\EBord(M)$ determines maps
\beq
&&1|1\EFT(M)\stackrel{{\rm restr}}{\to} C^\infty(\mathcal{L}^{1|1}_0(M))^{\Euc_{1|1}}\stackrel{{\rm cocycle}}{\twoheadrightarrow} \H(M;\C[\beta,\beta^{-1}]),\qquad |\beta|=-2,  \label{eq:thm1}\\
&&2|1\EFT(M)\stackrel{{\rm restr}}{\to} C^\infty(\mathcal{L}^{2|1}_0(M))^{\Euc_{2|1}}\stackrel{{\rm cocycle}}{\twoheadrightarrow} \H(M;\MF)\label{eq:thm2}
\eeq
from field theories to these cohomology theories over $\C$. 
\end{thm}

For $M=\pt$, the map~\eqref{eq:thm2} specializes to part of an announced result of Stolz and Teichner~\cite[Theorem~1.15]{ST11}; see Remark~\ref{rmk:STannounce}. Applied to general manifolds~$M$, one can identify $\H(-;\C[\beta,\beta^{-1}])$ with complexified K-theory, and $\H(-;\MF)$ with a version of $\TMF$ over $\C$; see~\S\ref{sec:MF}. Hence, Theorem~\ref{thm} proves a version of the conjectures~\eqref{eq:conjecture} over$~\C$. 

We elaborate on this connection between Theorem~\ref{thm} and the conjectures~\eqref{eq:conjecture}. The maps~\eqref{eq:thm1} and~\eqref{eq:thm2} come from sending a field theory to its partition function. This assignment defines a type of character map for field theories. Similarly, the cohomology theories in~\eqref{eq:conjecture} have Chern characters valued in certain cohomology theories defined over~$\C$. Putting these ingredients together, we obtain the diagrams 
%
\beq
&&\resizebox{.95\textwidth}{!}{
\begin{tikzpicture}[baseline=(basepoint)];
\node (A) at (0,0) {$1|1\EFT(M)$};
\node (B) at (4,0) {$\K(M)$};
\node (C) at (0,-1.25) {$C^\infty(\mathcal{L}_0^{1|1}(M))^{\Euc_{1|1}}$};
\node (D) at (4,-1.25) {$\H(M;\C[\beta,\beta^{-1}])$};
\draw[->>,dashed] (A) to node [above] {cocycle}  (B);
\draw[->] (A) to node [left] {${\rm restr}$} (C);
\draw[->] (B) to node [right] {Ch}  (D);
\draw[->>] (C) to node [above] {cocycle} (D);
\path (0,-.75) coordinate (basepoint);
\end{tikzpicture}\quad 
\begin{tikzpicture}[baseline=(basepoint)];
\node (A) at (0,0) {$2|1\EFT(M)$};
\node (B) at (4,0) {$\TMF(M)$};
\node (C) at (0,-1.25) {$C^\infty(\mathcal{L}_0^{2|1}(M))^{\Euc_{2|1}}$};
\node (D) at (4,-1.25) {$\H(M;\MF).$};
\draw[->>,dashed] (A) to node [above] {cocycle}  (B);
\draw[->] (A) to node [left] {${\rm restr}$} (C);
\draw[->] (B) to node [right] {Ch}  (D);
\draw[->>] (C) to node [above] {cocycle} (D);
\path (0,-.75) coordinate (basepoint);
\end{tikzpicture}}\label{eq:somesquares}
\eeq
One expects the cocycle maps in~\eqref{eq:conjecture} will make these diagrams commute. This offers new perspective on the conjectures~\eqref{eq:conjecture}, as we briefly summarize. Extending a partition function to a full field theory requires both additional data and property: a choice of preimage under the map ${\rm restr}$ in~\eqref{eq:thm1} and~\eqref{eq:thm2} need not exist nor be unique. Similarly, refining a cohomology class over~$\C$ to a class in the target of~\eqref{eq:conjecture} is both data and property: a class is in the image of the Chern character if it satisfies an integrality condition, and lifts of integral classes need not be unique owing to the presence of torsion. Up to an equivalence relation called \emph{concordance} (see below), the conjectures~\eqref{eq:conjecture} assert that the data and property determining such refinements---either as field theories or cohomology classes---precisely match each other. 


The concordance relation features in the full conjecture of Stolz and Teichner, which asserts that the cocycle maps~\eqref{eq:conjecture} induce bijections between concordance classes of field theories and cohomology classes~\cite[\S1.5-1.6]{ST11}. Recall that for a sheaf $\F\colon {\sf Mfld}^\op\to {\sf Set}$ on the site of manifolds, sections $s_0,s_1\in \F(M)$ are \emph{concordant} if there exists $s\in \F(M\times \R)$ such that $s_0= i_0^*s$, $s_1= i_1^*s$ where $i_0,i_1\colon M\hookrightarrow M\times \R$ are the inclusions at~0 and~1. Concordance defines an equivalence relation on the set $\F(M)$ whose equivalence classes are \emph{concordance classes}.

\begin{prop} \label{mainprop}
The assignment $M\mapsto C^\infty(\mathcal{L}^{d|1}_0(M))^{\Euc_{d|1}}$ is a sheaf on the site of manifolds. Concordance classes of sections map surjectively to $\H(M;\C[\beta,\beta^{-1}])$ and $\H(M;\MF)$ when $d=1$ and~$2$, respectively.
\end{prop}

There is an analogous definition of concordance for (higher) stacks, where the stack condition is used to show that the concordance relation is transitive. Assuming that $M\mapsto d|1\EFT(M)$ is a $d$-stack, Proposition~\ref{mainprop} implies that concordance classes of $d|1$-dimensional Euclidean field theories map to $\H(M;\C[\beta,\beta^{-1}])$ and $\H(M;\MF)$ for $d=1$ and 2, respectively. We expect this to implement the Chern character for K-theory and TMF through the maps on concordance classes induced by the diagrams~\eqref{eq:somesquares}.

This brings us to a technical point: although it is expected that the assignment $M\mapsto d|1\EFT(M)$ is a $d$-stack, when $d=2$ this statement is contingent on a fully-extended enhancement of the existing definitions. This fully-extended aspect is an essential ingredient in Stolz and Teichner's conjecture that concordance classes of $2|1$-dimensional field theories yield TMF; see~\cite[Conjecture~1.17]{ST11}. In this paper, the source of~\eqref{eq:thm2} uses the 1-categorical definition from~\cite{ST11}. Fully-extended $2|1$-dimensional super Euclidean field theories should map to this 1-categorical version (via a forgetful functor), and from this one would obtain a Chern character on concordance classes via post-composition with~\eqref{eq:thm2}.

\subsection{Cocycles from partition functions}
In physics, the best-known topological invariants associated with the field theories~\eqref{eq:defnFT} are the Witten index in dimension~$1|1$~\cite{susymorse}, and the elliptic genus in dimension~$2|1$~\cite{Witten_Elliptic,Alvarezetal}. These are examples of \emph{partition functions}. For example, when $d=2$ the partition function of the $\mathcal{N}=(0,1)$ supersymmetric sigma model with target a string manifold is a modular form called the \emph{Witten genus}~\cite{witten_dirac}. This genus led Segal to suggest that certain 2-dimensional quantum field theories could provide a geometric model for elliptic cohomology~\cite{Segal_Elliptic}. 

Stolz and Teichner refined these early ideas, leading to the conjectured cocycle maps~\eqref{eq:conjecture}. In their framework (as in Segal's~\cite{Segal_cft}), partition functions are defined as the value of a field theory on closed, connected bordisms~\cite[Definition~4.13]{ST11}. The definition of a super Euclidean field theory implies that this restriction determines a function invariant under the action by super Euclidean isometries
\beq
&&d|1\EFT(M)\to C^\infty(\{{\rm closed\ bordisms\ over \ } M\})^{\rm isometries.}\label{eq:partition}
\eeq
Fei Han~\cite{Han} shows that~\eqref{eq:partition} applied to a class of $1|1$-dimensional closed bordisms over~$M$
\beq
\Map(\R^{0|1},M)\simeq \Map(\R^{1|1}/\Z,M)^{S^1}\subset \Map(\R^{1|1}/\Z,M)\subset 1|1\Bord(M)\label{eq:Hanres}
\eeq
encodes the Chern character in K-theory. To summarize, restriction along~\eqref{eq:Hanres} evaluates a $1|1$-dimensional Euclidean field theory on length~1 super circles whose map to~$M$ is invariant under the action of loop rotation. This restriction is also a version of \emph{dimensional reduction}. When the input $1|1$-dimensional Euclidean field theory is constructed via Dumestrcu's super parallel transport for a vector bundle with connection~\cite{florin}, the resulting element in $C^\infty(\Map(\R^{0|1},M))\simeq \Omega^\bullet(M)$ is a differential form representative of the Chern character of that vector bundle. 

The cocycle map~\eqref{eq:thm2} is a more elaborate version of restriction along~\eqref{eq:Hanres}. The goal is to find an appropriate class of closed $2|1$-dimensional bordisms so that the restriction~\eqref{eq:partition} constructs a map from $2|1$-dimensional Euclidean field theories to complex analytic elliptic cohomology. There are two main problems to be solved in this 2-dimensional generalization. First, one cannot specialize to a particular super torus, as in the specialization to the length~1 super circle in~\eqref{eq:Hanres}. Indeed, elliptic cohomology over $\C$ is parameterized by the moduli of all complex analytic elliptic curves. This problem is easy enough to solve, though its resolution introduces some technicalities: one restricts to a moduli \emph{stack} of super tori. 

The second obstacle is more serious. Stolz and Teichner's field theories are neither chiral nor conformal, and hence restriction only gives a \emph{smooth} function on the moduli stack of super Euclidean tori.
On the other hand, a class in complex analytic elliptic cohomology only depends on the holomorphic part of the conformal modulus of a torus. 
Resolving this apparent mismatch comes through a surprising feature of the super moduli space~$\mathcal{L}^{2|1}_0(M)$: the failure of conformality and holomorphy is measured by a specified de~Rham coboundary (see Proposition~\ref{prop:compute21}). Loosely, this shows that functions on~$\mathcal{L}^{2|1}_0(M)$ possess a kind of \emph{derived} holomorphy and conformality. 


\subsection{Outline of the proof} \label{sec:outline}
Theorem~\ref{thm} boils down to somewhat technical computations in supermanifolds, so we briefly outline the approach and state key intermediary results in terms of ordinary (non-super) geometry. There are three main steps in the construction.
\begin{enumerate}
\item[(i)] Construct the super moduli spaces $\mathcal{L}^{d|1}_0(M)$. 
\item[(ii)] Compute the algebras of $\Euc_{d|1}$-invariant functions $C^\infty(\mathcal{L}^{d|1}_0(M))^{\Euc_{d|1}}$ in terms of differential form data on $M$. 
\item[(iii)] Construct the cocycle maps~\eqref{eq:thm1} and~\eqref{eq:thm2} using the output of  step (ii).
\end{enumerate}
The main work is in step (ii), culminating in Propositions~\ref{prop:compute11} and~\ref{prop:compute21} below.

For step (i), we start by defining
\beq
\mathcal{L}^{d|1}(M)&:=&\mathcal{M}^{d|1}\times \Map(\R^{d|1}/\Z^d,M),\qquad \mathcal{L}^{d|1}(M)\subset d|1\Bord(M),\label{eq:Ld1def}
\eeq
where $\mathcal{M}^{d|1}$ is the moduli space of super Euclidean structures on~$\R^{d|1}/\Z^d$, and $\Map(\R^{d|1}/\Z^d,M)$ is the generalized supermanifold of maps from $\R^{d|1}/\Z^d$ to $M$. Hence, an $S$-point of $\mathcal{L}^{1|1}(M)$ determines a family of super Euclidean circles with a map to~$M$, and an $S$-point of $\mathcal{L}^{2|1}(M)$ determines a family of super Euclidean tori with a map to~$M$. There is a canonical functor $\mathcal{L}^{d|1}(M)\to d|1\Bord(M)$, regarding these supermanifolds as bordisms from the empty set to the empty set.
%
Next we consider the subobject of~\eqref{eq:Ld1def} gotten by taking maps invariant under the $\R^d$-action on $\R^{d|1}/\Z^d$ by precomposition. Equivalently, this is the $S^1=\R/\Z$-fixed subspace when $d=1$ and the $T^2=\R^2/\Z^2$-fixed subspace when $d=2$. This yields finite-dimensional subobjects, 
\beq
\mathcal{M}^{d|1}\times \Map(\R^{0|1},M) \simeq \mathcal{L}^{d|1}_0(M):=\mathcal{L}^{d|1}(M)^{\R^d/\Z^d}\subset \mathcal{L}^{d|1}(M),\label{eq:inclusion}
\eeq
that, roughly speaking, are the subspaces of maps that are constant up to nilpotents. Restricting a field theory along the composition $\mathcal{L}^{d|1}_0(M)\subset \mathcal{L}^{d|1}(M)\to d|1\Bord(M)$ extracts a function, providing the first arrow in~\eqref{eq:thm1} and~\eqref{eq:thm2} (see Lemmas~\ref{lem:11partiton} and~\ref{lem:21partiton})
\beq
{\rm restr}\colon d|1\EFT(M)\to C^\infty(\mathcal{L}^{d|1}_0(M))^{\Euc_{d|1}}.\label{eq:partitionfunction}
\eeq

\begin{rmk}
The restriction~\eqref{eq:partitionfunction} is dimensional reduction in the sense of~\cite[Glossary]{strings1}, though it differs from dimensional reduction in the sense of~\cite[\S1.3]{ST11}. 
\end{rmk}

Step (ii) is a technical computation. The $d=1$ case is characterized as follows.

\begin{prop} \label{prop:compute11}
Elements of $C^\infty(\mathcal{L}^{1|1}_0(M))^{\Euc_{1|1}}$ are in bijection with pairs $(Z,Z_\ell)$, 
\beq
&&Z\in (\Omega^\bullet_\cl(M;C^\infty(\R_{>0})[\beta,\beta^{-1}]))^0 \qquad |\beta|=-2\label{eq:itsanidentification}\\
&&Z_\ell\in (\Omega^\bullet(M;C^\infty(\R_{>0})[\beta,\beta^{-1}]))^{-1}\nonumber
\eeq
where $Z$ is closed of total degree zero, $Z_\ell$ is of total degree $-1$, and they satisfy 
\beq
&&\partial_\ell Z=\dR Z_\ell,\label{eq:adddata1}
\eeq
where $\dR$ is the de~Rham differential on~$M$, and $\partial_\ell$ is the vector field on $\R_{>0}$ associated to the standard coordinate $\ell\in C^\infty(\R_{>0})$.
\end{prop}

For the $d=2$ case, let $\HH\subset \C$ denote the upper half plane with standard complex coordinates $\tau,\bar\tau\in C^\infty(\HH)$, and let $v\in C^\infty(\R_{>0})$ be the standard coordinate.

\begin{prop} \label{prop:compute21}
Elements of $C^\infty(\mathcal{L}^{2|1}_0(M))^{\Euc_{2|1}}$ are in bijection with triples $(Z,Z_{\bar \tau},Z_v)$
\beq
&&Z\in (\Omega^\bullet_\cl(M;C^\infty(\HH\times \R_{>0})[\beta,\beta^{-1}])^{\SL_2(\Z)})^0\qquad |\beta|=-2\label{eq:itsanidentification2}\\
&&Z_{\bar\tau},Z_v\in (\Omega^{\bullet}(M;C^\infty(\HH\times \R_{>0})[\beta,\beta^{-1}]))^{-1}\nonumber
\eeq
where $Z$ is closed of total degree zero, $Z_{\bar \tau},Z_v$ are of total degree $-1$, they satisfy an $\SL_2(\Z)$-invariance property stated in Lemma~\ref{lem:SL2}, and 
\beq
&&\partial_v Z=\dR Z_v,\qquad \partial_{\bar \tau} Z=\dR Z_{\bar\tau},\label{eq:adddata2}
\eeq
where $\dR$ is the de~Rham differential on~$M$, and $\partial_{\bar\tau}$, $\partial_v$ are vector fields on $\HH$ and $\R_{>0}$.
\end{prop}

In Propositions~\ref{prop:compute11} and~\ref{prop:compute21}, the closed differential form $Z$ arises by restriction to a subspace 
\beq
\R_{>0}\times \Map(\R^{0|1},M)&\hookrightarrow& \mathcal{L}^{1|1}_0(M), \label{eq:equivinc1}\\
\Lat \times \Map(\R^{0|1},M)&\hookrightarrow& \mathcal{L}^{2|1}_0(M),\label{eq:equivinc2}
\eeq
where $\Lat$ is the space of based, oriented lattices in $\C$. Indeed,~\eqref{eq:itsanidentification} and~\eqref{eq:itsanidentification2} come from 
 \beq\label{Eq:diffformdata}
&&\begin{array}{lcl} C^\infty(\R_{>0}\times \Map(\R^{0|1},M))^{\Euc_{1|1}}&\simeq& (\Omega_\cl(M;C^\infty(\R_{>0})[\beta,\beta^{-1}]))^0, \\
C^\infty(\Lat \times \Map(\R^{0|1},M))^{\Euc_{2|1}}&\simeq& (\Omega^\bullet(M;C^\infty(\HH\times \R_{>0})[\beta,\beta^{-1}])^{\SL_2(\Z)})^0.
\end{array}
\eeq
When $d=1$, $\ell\in \R_{>0}$ corresponds to (super) circles of length $\ell$, and~\eqref{eq:adddata1} shows that the failure of $Z\in C^\infty(\R_{>0}\times \Map(\R^{0|1},M))^{\Euc_{1|1}}$ to be independent of this length is $\dR$-exact. When $d=2$, a point $(\tau,\bar\tau,v)\in \HH\times \R_{>0}$ corresponds to (super) Euclidean tori with conformal modulus $(\tau,\bar\tau)$ and total volume $v$. Then $Z_v$ and~\eqref{eq:adddata2} show that the failure of $Z\in C^\infty(\Lat \times \Map(\R^{0|1},M))^{\Euc_{2|1}}$ to be independent of the volume is $\dR$-exact. Similarly, $Z_{\bar \tau}$ and~\eqref{eq:adddata2} show that the failure of $Z$ to depend holomorphically on the conformal modulus is $\dR$-exact. This is the precise sense in which functions on~$\mathcal{L}^{2|1}_0(M)$ exhibit a derived version of holomorphy and conformality.

Finally for step (iii), we consider the maps (with notation from Propositions~\ref{prop:compute11} and~\ref{prop:compute21})
\beq
&&C^\infty(\mathcal{L}^{1|1}_0(M))\to  \H(M;C^\infty(\R_{>0})[\beta,\beta^{-1}]), \qquad (Z,Z_\ell)\mapsto [Z]\label{eq:cocyclesp1}\\
&&C^\infty(\mathcal{L}^{2|1}_0(M))\to  \H(M;C^\infty(\HH\times \R_{>0})[\beta,\beta^{-1}])^{\SL_2(\Z)},\qquad  (Z,Z_{\bar \tau},Z_\ell)\mapsto [Z] \label{eq:cocyclesp2},
\eeq
where $|\beta|=-2$ and has weight~1 for $\SL_2(\Z)$, meaning $\beta\mapsto (c\tau+d)\beta$.

\begin{proof}[Proof of Theorem~\ref{thm} from Propositions~\ref{prop:compute11} and~\ref{prop:compute21}] Starting with the $d=1$ case, we claim that the map~\eqref{eq:cocyclesp1} factors through cohomology with coefficients in the subring $\C[\beta,\beta^{-1}]\hookrightarrow C^\infty(\R_{>0})[\beta,\beta^{-1}]$, including as the constant functions on $\R_{>0}$. Indeed, observe
\beq
&&\partial_\ell [Z]=[\partial_\ell Z]=[dZ_\ell] =0,\label{eq:constantinell}
\eeq
using~\eqref{eq:adddata1}. Hence, $[Z]\in \H(M;\C[\beta,\beta^{-1}])\subset \H(M;C^\infty(\R_{>0})[\beta,\beta^{-1}])$ and~\eqref{eq:cocyclesp1} determines the cocycle map in~\eqref{eq:thm1}. 

Similarly, the map~\eqref{eq:cocyclesp2} factors through cohomology with coefficients in the subring 
\beq
&&\MF\simeq (\mathcal{O}(\HH)[\beta,\beta^{-1}])^{\SL_2(\Z)}\hookrightarrow (C^\infty(\HH\times \R_{>0})[\beta,\beta^{-1}])^{\SL_2(\Z)},\label{eq:itsaninclusion}
\eeq
where $\MF$ is the ring of weak modular forms (see Definition~\ref{defn:mf}). The map~\eqref{eq:itsaninclusion} is the pullback of smooth functions along the projection $\HH\times \R_{>0}\to \HH$ composed with the inclusion $\mathcal{O}(\HH)\subset C^\infty(\HH)$. Indeed, we have
$$
\partial_v [Z]=[\partial_v Z]=[dZ_v] =0,\qquad \partial_{\bar \tau} [Z]=[\partial_{\bar \tau} Z]=[dZ_{\bar \tau}] =0,
$$
using~\eqref{eq:adddata2}, where the first set of equalities demonstrate independence from $\R_{>0}$, while the second demonstrate holomorphic dependence on $\HH$. Finally, the $\SL_2(\Z)$-invariance property for $Z$ (see Lemma~\ref{lem:SL2}) shows that~$[Z]$ is indeed a cohomology class valued in modular forms,
$$
[Z]\in \H(M;\MF)\subset \H(M;C^\infty(\HH\times \R_{>0})[\beta,\beta^{-1}])^{\SL_2(\Z)}
$$
and hence~\eqref{eq:cocyclesp2} determines the cocycle map in~\eqref{eq:thm2}. 

Surjectivity of the cocycle maps~\eqref{eq:thm1} and~\eqref{eq:thm2} follows from the inclusions
\beq
\Omega^\bullet_{\rm cl}(M;\C[\beta,\beta^{-1}])&\hookrightarrow& C^\infty(\mathcal{L}^{1|1}_0(M)),\qquad \omega\mapsto (\omega,0)=(Z,Z_\ell)\nonumber\\
\Omega^\bullet_{\rm cl}(M;\MF)&\hookrightarrow& C^\infty(\mathcal{L}^{2|1}_0(M)),\qquad \omega\mapsto (\omega,0,0)=(Z,Z_{\bar \tau},Z_v),\nonumber
\eeq
using the description of functions from Propositions~\ref{prop:compute11} and~\ref{prop:compute21} and the maps on coefficients described in the previous two paragraphs. The definition of the maps~\eqref{eq:cocyclesp1} and~\eqref{eq:cocyclesp2} together with the de~Rham theorem then implies that every cohomology class admits a refinement to a function on~$\mathcal{L}^{d|1}_0(M)$.
\ep


The following remarks relate our results to other work.

\begin{rmk}
The above analysis of the moduli space of super Euclidean tori is related to previous investigations of moduli spaces of super Riemann surfaces in the string theory literature, e.g., see~\cite{notprojected,WittensuperRiem}. However, the vast majority of prior constructions in string theory and in the Stolz--Teichner program only study the reduced moduli spaces. In particular, the cocycle models for (equivariant) elliptic cohomology in~\cite{DBE_EquivTMF,DBE_MQ,Powerops,BET1} arise as functions on the reduced moduli space. In this prior work, the correct mathematical object comes only after imposing holomorphy by hand. However, as Theorem~\ref{thm} shows, this property emerges naturally from the geometry of $2|1$-dimensional super tori. 
\end{rmk} 

\begin{rmk}
When $M=\pt$, Proposition~\ref{prop:compute21} shows that partition functions of $\mathcal{N}=(0,1)$ supersymmetric quantum field theories are weak modular forms: since $\Omega^\odd(\pt)=\{0\}$, $Z_v=Z_{\bar \tau}=0$ are no additional data. In contrast to the arguments in the physics literature that analyze a particular action functional (e.g.,~\cite[\S4.3-4.4]{DualityMock}), the proof here emerges entirely from the geometry of the moduli space of super Euclidean tori.
This recovers Stolz and Teichner's claim from~\cite[pg.~10]{ST11} that ``holomorphicity is a consequence of the more intricate structure of the moduli stack of supertori." 
\end{rmk}

\begin{rmk}
The data $Z_{\bar \tau}$ in Proposition~\ref{prop:compute21} is closely related to \emph{anomaly cancelation} in physics and choices of \emph{string structures} in geometry. An illustrative example is the elliptic Euler class: an oriented vector bundle~$V\to M$ determines a class $[\Eu(V)]\in \H(M;\MF)$ if the Pontryagin class $[p_1(V)]\in \H^4(M;\R)$ vanishes. In~\S\ref{sec:Witten} we show that the set of differential forms $H\in \Omega^3(M;\R)$ with $p_1(V)=\dR H$ parameterizes cocycle refinements of $[\Eu(V)]$ to a function on $\mathcal{L}^{2|1}_0(M)$.
Geometrically,~$H$ is part of the data of a string structure on~$V$. In physics,~$H$ is part of the data for anomaly cancelation in a theory of $V$-valued free fermions. Under the conjectured cocycle maps~\eqref{eq:conjecture} $V$-valued free fermions are expected to furnish representatives of elliptic Euler classes in $\TMF(M)$~\cite[\S4.4]{ST04}. Perturbative quantization of fermions rigorously constructs elliptic Euler cocycles over~$\C$~\cite[\S6]{DBE_MQ}, and Theorem~\ref{thm} shows that lifting a cohomology class to a $2|1$-dimensional Euclidean field theory must depend on a \emph{choice} of string structure, at least rationally. 
\end{rmk}

\begin{rmk}
If the input field theory in~\eqref{eq:thm2} is super conformal, then~$\dR Z_v=0$, whereas if the input theory is holomoprhic then $\dR Z_{\bar\tau}=0$. For a general field theory (not necessarily conformal or holomorphic) the differential form $\partial_{\bar \tau}Z_\ell-\partial_\ell Z_{\bar \tau}$ is closed. These closed forms have the potential to encode secondary cohomological invariants of field theories. Although we do not know explicit field theories for which this cohomology class is nonzero, the structure appears to be related to mock modular phenomena and the $\TMF$-valued torsion invariants studied in~\cite{GJFW,TheoDavide}. 
\end{rmk}


\begin{rmk}
In light of Fei Han's work on the Bismut--Chern character~\cite{Han}, it is tempting to think of the restriction $2|1\EFT(M)\to C^\infty(\mathcal{L}^{2|1}(M))^{\Euc_{2|1}}$ (without taking $T^2$-invariant maps) as a candidate construction of the \emph{elliptic} Bismut--Chern character. Indeed, functions on $C^\infty(\mathcal{L}^{2|1}(M))^{\Euc_{2|1}}$ can be identified with cocycles analogous to~\eqref{eq:adddata2}, where $Z$ is a differential form on the double loop space and the de~Rham differential $\dR$ is replaced with the $T^2$-equivariant differential investigated in~\cite{BElocalization}. 
\end{rmk}


\subsection{Conventions for supermanifolds}\label{sec:conventions} This paper works in the category of supermanifolds with structure sheaves defined over~$\C$; this is called the category of $cs$-supermanifolds in~\cite{DM}. The majority of what we require is covered in the concise introduction~\cite[\S4.1]{ST11}, but we establish a little notation presently. First, all functions and differential forms are $\C$-valued. The supermanifolds~$\R^{n|m}$ are characterized by their super algebra of functions $C^\infty(\R^{n|m})\simeq C^\infty(\R^n;\C)\otimes_\C \Lambda^\bullet \C^m$. The representable presheaf associated with~$\R^{n|m}$ assigns to a supermanifold $S$ the set 
$$
\R^{n|m}(S):=\{t_1,t_2,\dots,t_n\in C^\infty(S)^{\ev}, \theta_1,\theta_2,\dots,\theta_m\in C^\infty(S)^\odd \mid (t_i)_{\rm red}=\overline{(t_i)}_{\rm red}\}
$$
where $(t_i)_{\rm red}$ denotes the restriction of a function to the reduced manifold $S_{\rm red}\hookrightarrow S$, and $\overline{(t_i)}_{\rm red}$ is the conjugate of the complex-valued function $(t_i)_{\rm red}$ on the smooth manifold~$S_{\rm red}$. We use this functor of points description throughout the paper, typically with roman letters denoting even functions and greek letters denoting odd functions. 

We follow Stolz and Teichner's terminology wherein a presheaf on supermanifolds is called a \emph{generalized supermanifold}. An example of a generalized supermanifold is $\Map(X,Y)$ for supermanifolds $X$ and $Y$ that assigns to a supermanifold $S$ the set of maps $S\times X\to Y$. For a manifold $M$ regarded as a supermanifold, the generalized supermanifold $\Map(\R^{0|1},M)$ is isomorphic to the representable presheaf associated to the odd tangent bundle $\Pi TM$, as we recall briefly. We use the notation $(x,\psi)\in \Pi TM(S)$ for an $S$-point, where $x\colon S\to M$ is a map and $\psi\in \Gamma(S;x^* TM)^\odd$ is an odd section. This gives an $S$-point $(x+\theta\psi)\in \Map(\R^{0|1},M)$ by identifying $x$ with an algebra map $x\colon C^\infty(M)\to C^\infty(S)$ and $\psi\colon C^\infty(M)\to C^\infty(S)$ with an odd derivation relative to~$x$. These fit together to define an algebra map 
\beq
C^\infty(M)\stackrel{(x,\psi)}{\to}C^\infty(S)\oplus \theta \cdot C^\infty(S) \simeq C^\infty(S\times \R^{0|1}),\label{eq:algebramap}
\eeq
with the isomorphism coming from Taylor expansion in a choice of odd coordinate $\theta\in C^\infty(\R^{0|1})$. The map~\eqref{eq:algebramap} is equivalent to $S\times\R^{0|1}\to M$, i.e., an $S$-point of $\Map(\R^{0|1},M)$. 
The functions $C^\infty(\Map(\R^{0|1},M))\simeq C^\infty(\Pi TM)\simeq \Omega^\bullet(M)$ recover differential forms on $M$ as a $\Z/2$-graded $\C$-algebra. The action of automorphisms of $\R^{0|1}$ on this algebra encode the de~Rham differential and the grading operator on forms; e.g., see~\cite[\S3]{HKST}. 

\subsection{Acknowledgements} It is a pleasure to thank Bertram Arnold, Theo Johnson-Freyd, Stephan Stolz, Peter Teichner, and Arnav Tripathy for fruitful conversations that shaped this work, as well as a referee for finding a mistake in a previous draft and making several suggestions that improved the clarity and precision of the paper.

\section{A map from $1|1$-Euclidean field theories to complexified K-theory}\label{sec:Kthy}

The main goal of this section is to prove Proposition~\ref{prop:compute11}. From the discussion in~\S\ref{sec:outline}, this proves the $d=1$ case of Theorem~\ref{thm}. We also prove Proposition~\ref{mainprop} when $d=1$, and connect this result with Chern characters of super connections.

\subsection{The moduli space of super Euclidean circles}

\begin{defn}
Let~$\E^{1|1}$ denote the super Lie group with underlying supermanifold $\R^{1|1}$ and multiplication
\beq
(t,\theta)\cdot (t',\theta')=(t+t'+i\theta\theta',\theta+\theta'), \quad (t,\theta),(t',\theta')\in \R^{1|1}(S).\label{eq:E11mult}
\eeq
Define the \emph{super Euclidean group} as $\Euc_{1|1}:=\E^{1|1}\rtimes \Z/2$ where the semidirect product is defined using the $\Z/2=\{\pm 1\}$-action by reflection $
(t,\theta)\mapsto (t,\pm\theta)$, for $(t,\theta)\in\E^{1|1}(S).$
\end{defn} 

The super Lie algebra of $\E^{1|1}$ is generated by a single odd element, namely the left invariant vector field $D=\partial_\theta-i\theta\partial_t$. The right-invariant generator is $Q=\partial_\theta+i\theta\partial_t$. The super commutators are
\beq
\frac{1}{2}[D,D]=D^2=-i\partial_t,\qquad \frac{1}{2}[Q,Q]=Q^2=i\partial_t.\label{eq:Wick}
\eeq

\begin{rmk} The factors of $i=\sqrt{-1}$ in~\eqref{eq:E11mult} and~\eqref{eq:Wick} come from Wick rotation, e.g., see~\cite[pg.~95, Example~4.9.3]{strings1}. This differs from the convention for the $1|1$-dimensional Euclidean group in~\cite[Definition~33]{HST}, but is more closely aligned with the Wick rotated $2|1$-dimensional Euclidean geometry defined in~\cite[\S4.2]{ST11} and studied below. \end{rmk} 

Let $\R^{1|1}_{>0}$ denote the supermanifold gotten by restricting the structure sheaf of $\R^{1|1}$ to the positive reals,~$\R_{>0}\subset \R$.


\begin{defn}\label{defn:supercircle}
Given an $S$-point $(\ell,\lambda)\in \R^{1|1}_{>0}(S)$, define the family of $1|1$-dimensional \emph{super Euclidean circles} as the quotient
\beq
S^{1|1}_{\ell,\lambda}:=(S\times \R^{1|1})/\Z\label{eq:supercircle}
\eeq
for the left $\Z$-action over $S$ determined by the 
formula
\beq
n\cdot (t,\theta)= (t+n\ell+in\lambda\theta,n\lambda+\theta),\quad n\in \Z(S), (t,\theta)\in \R^{1|1}(S).\label{eq:Zact}
\eeq
Equivalently this is the restriction of the left $\E^{1|1}$-action on $S\times \R^{1|1}$ to the $S$-family of subgroups $\Z\times S\subset \E^{1|1}\times S$ with generator $\{1\}\times S\simeq S\stackrel{(\ell,\lambda)}{\hookrightarrow} \R^{1|1}_{>0}\times S\subset \E^{1|1}\times S$. Define the \emph{standard super Euclidean circle} denoted $S^{1|1}=S^{1|1}_{1,0}=\R^{1|1}/\Z$ as the quotient by the action for the standard inclusion $\Z\subset \R\subset \E^{1|1}$. 
\end{defn}

\begin{rmk}
The $S$-family of subgroups $S\times \Z\hookrightarrow S\times \E^{1|1}$ generated by $(\ell,\lambda)\in \R^{1|1}_{>0}(S)$ is normal if and only if $\lambda=0$. Hence, the standard super circle $S^{1|1}$ inherits a group structure from $\E^{1|1}$, but a generic $S$-family of super Euclidean circles $S^{1|1}_{\ell,\lambda}$ does not. 
\end{rmk}

\begin{rmk} There is a more general notion of a family of super circles where~\eqref{eq:Zact} incorporates the action by~$\Z/2<\Euc_{1|1}$. This moduli space has two connected components corresponding to choices of spin structure on the underlying ordinary circle, with the component from Definition~\ref{defn:supercircle} corresponding to the odd (or nonbounding) spin structure. This turns out to be the relevant component to recover complexified K-theory. 
\end{rmk}

We recall~\cite[Definitions~2.26, 2.33 and 4.4]{ST11} that for a supermanifold $\mathbb{M}$ with an action by a super Lie group $G$, an \emph{$(\mathbb{M},G)$-structure} on a family of supermanifolds $T\to S$ is an open cover $\{U_i\}$ of $T$ with isomorphisms to open sub supermanifolds $\varphi\colon U_i\stackrel{\sim}{\to} V_i\subset S\times \mathbb{M}$ and transition data $g_{ij}\colon V_i\bigcap V_j\to G$ compatible with the $\varphi_i$ and satisfying a cocycle condition. An \emph{isometry} between supermanifolds with $(\mathbb{M},G)$-structure is defined as a map $T\to T'$ over $S$ that is locally given by the $G$-action on~$\mathbb{M}$, relative to the open covers~$\{U_i\}$ of~$T$ and~$\{U_i'\}$ of~$T'$. Supermanifolds with $(\mathbb{M},G)$-structure and isometries form a category fibered over supermanifolds. 

\begin{defn}[\cite{ST11} \S4.2]\label{defn:11superEuc}
A \emph{super Euclidean} structure on a $1|1$-dimensional family $T\to S$ is an $(\mathbb{M},G)$-structure for the left action of $G=\Euc_{1|1}$ on~$\mathbb{M}=\R^{1|1}$. 
\end{defn}


\begin{lem} \label{lem:supercircleEuc}
An $S$-family of super circles~\eqref{eq:supercircle} has a canonical super Euclidean structure.  
\end{lem} 

\bp
We endow a family of super circles with a $1|1$-dimensional Euclidean structure as follows. Take the open cover $S\times \R^{1|1}\to S^{1|1}_{\ell,\lambda}$ supplied by the quotient map, and take transition data from the $\Z$-action on $S\times \R^{1|1}$. By definition this $\Z$-action is through super Euclidean isometries, and so the quotient inherits a super Euclidean structure. 
\ep

We observe that every family of super circles pulls back from the universal family $(\R^{1|1}_{>0}\times \R^{1|1})/\Z\to \R^{1|1}_{>0}$ along a map $S\to \R^{1|1}_{>0}$. Hence, 
$$
\mathcal{M}^{1|1}:=\R^{1|1}_{>0},\qquad \mathcal{S}^{1|1}:=(\R^{1|1}_{>0}\times \R^{1|1})/\Z\to \R^{1|1}_{>0},
$$
is the moduli space of super Euclidean circles and the universal family of super Euclidean circles, respectively. The following shows that $\mathcal{M}^{1|1}=\R^{1|1}_{>0}$ can equivalently be viewed as the moduli space of super Euclidean structures on the standard super circle. 

\begin{lem} \label{lem:torusstnd} There exists an isomorphism of supermanifolds over~$\R^{1|1}_{>0}$,
\beq
&&\R^{1|1}_{>0}\times S^{1|1}\stackrel{\sim}{\to} \mathcal{S}^{1|1}, \label{eq:isotostdS}
\eeq
from the constant $\R^{1|1}_{>0}$-family with fiber the standard super circle to the universal family of super circles. This isomorphism does not preserve the super Euclidean structure on~$\mathcal{S}^{1|1}$. 
\end{lem}
\bp
Define the map 
\beq
\R^{1|1}_{>0}\times \R^{1|1}&\to& \R^{1|1}_{>0}\times \R^{1|1}\label{eq:Zequiv}\\
(\ell,\lambda,t,\theta)&\mapsto& (\ell,\lambda,t(\ell+i\lambda\theta),\theta+t\lambda),\qquad (\ell,\lambda)\in \R^{1|1}_{>0}(S),\ (t,\theta)\in \R^{1|1}(S).\nonumber
\eeq
Observe that~\eqref{eq:Zequiv} is $\Z$-equivariant for the action on the source and target given by
$$
n\cdot (\ell,\lambda,t,\theta)=(\ell,\lambda,t+n,\theta),\qquad n\cdot (\ell,\lambda,t,\theta)=(\ell,\lambda,t+n(\ell+i\lambda\theta),\theta+n\lambda)
$$
respectively. Hence~\eqref{eq:Zequiv} determines a map between the respective $\Z$-quotients, defining a map~\eqref{eq:isotostdS}. This is easily seen to be an isomorphism of supermanifolds. Since~\eqref{eq:isotostdS} is not locally determined by the action of $\Euc_{1|1}$ on $\R^{1|1}$, it is not a super Euclidean isometry. 
\ep

The following gives an $S$-point formula for the action of $\Euc_{1|1}$ on $\mathcal{S}^{1|1}$ and $\mathcal{M}^{1|1}=\R_{>0}^{1|1}$ coming from isometries between super Euclidean circles. 

\begin{lem}\label{lem:supercircleE11} Given $(\ell,\lambda)\in \R^{1|1}_{>0}(S)=\mathcal{M}^{1|1}(S)$ and $(s,\eta,\pm 1)\in (\E^{1|1}\rtimes \Z/2)(S)=\Euc_{1|1}(S)$, there is an isometry $f_{(s,\eta,\pm 1)}\colon S^{1|1}_{\ell,\lambda}\to S^{1|1}_{\ell',\lambda'}$ of super Euclidean circles over~$S$ sitting in the diagram
\beq
\begin{tikzpicture}[baseline=(basepoint)];
\node (A) at (0,0) {$S\times \R^{1|1}$};
\node (B) at (5,0) {$S\times \R^{1|1}$}; 
\node (C) at (0,-1.3) {$S^{1|1}_{\ell,\lambda}$};
\node (D) at (5,-1.3) {$S^{1|1}_{\ell',\lambda'}$};
\draw[->] (A) to node [above] {$(s,\eta,\pm 1)\cdot $} (B);
\draw[->] (A) to (C);
\draw[->] (C) to node [above] {$f_{(s,\eta,\pm 1)}$} (D);
\draw[->] (B) to  (D);
\path (0,-.75) coordinate (basepoint);
\end{tikzpicture}\nonumber
\eeq
where the upper horizontal arrow is determined by the left $\Euc_{1|1}$-action on $\R^{1|1}$, the left vertical arrow is the quotient map~\eqref{eq:supertorus} for $(\ell,\lambda)$, and the right vertical arrow is the quotient map for
\beq\label{eq:R11action}
(\ell',\lambda'):=(\ell\pm 2i\eta\lambda,\pm \lambda).
\eeq
\end{lem}

\bp
Consider the diagram
\beq
\begin{tikzpicture}[baseline=(basepoint)];
\node (A) at (0,0) {$\Z\times S\times \R^{1|1}$};
\node (B) at (5,0) {$\Z\times S\times \R^{1|1}$};
\node (C) at (0,-1.3) {$S\times \R^{1|1}$};
\node (D) at (5,-1.3) {$S\times \R^{1|1}$}; 
\draw[->] (A) to node [above] {$(s,\eta,\pm 1)\cdot $} (B);
\draw[->] (A) to node [left] {$(\ell,\lambda)\cdot $} (C);
\draw[->] (C) to node [below] {$(s,\eta,\pm 1)\cdot $} (D);
\draw[->] (B) to node [right] {$(\ell',\lambda')\cdot $} (D);
\path (0,-.75) coordinate (basepoint);
\end{tikzpicture}\label{eq:thediagram123}
\eeq
where the horizontal arrows denote the left action of $(s,\eta,\pm 1)\in \Euc_{1|1}(S)$ on $S\times \R^{1|1}$, while the vertical arrows denote the left $\Z$-action generated by $(\ell,\lambda),(\ell',\lambda')\in \R^{1|1}_{>0}(S)$. The square~\eqref{eq:thediagram123} commutes if and only if $(\ell',\lambda')=(s,\eta,\pm 1)\cdot (\ell, \lambda)\cdot (s,\eta,\pm 1)^{-1}\in \R^{1|1}_{>0}(S)\subset \E^{1|1}(S)$, i.e.,~\eqref{eq:R11action} holds. Commutativity of the diagram~\eqref{eq:thediagram123} gives a map on the $\Z$-quotients, which is precisely a map $S^{1|1}_{\ell,\lambda}\to S^{1|1}_{\ell',\lambda'}$. This map is locally determined by the action of $\E^{1|1}\rtimes \Z/2$, and hence respects the super Euclidean structures. 
\ep


\subsection{Super Euclidean loop spaces}

\begin{defn}\label{defn:sEucloop} Define the \emph{super Euclidean loop space} as the generalized supermanifold
$$
\mathcal{L}^{1|1}(M):=\R^{1|1}_{>0}\times \Map(S^{1|1},M).
$$ 
\end{defn}

We identify an $S$-point of $\mathcal{L}^{1|1}(M)$ with a map $S^{1|1}_{\ell,\lambda}\to M$ given by the composition 
\beq
S^{1|1}_{\ell,\lambda}\simeq  S\times S^{1|1}\to M\label{eq:superloop}
\eeq
by pulling back the isomorphism from Lemma~\ref{lem:torusstnd} along the map $(\ell,\lambda)\colon S\to \R^{1|1}_{>0}$.

We will define a left action of $\Euc_{1|1}$ on $\mathcal{L}^{1|1}(M)$ determined by the diagram
\beq
\begin{tikzpicture}[baseline=(basepoint)];
\node (A) at (0,0) {$S^{1|1}_{\ell,\lambda}$};
\node (B) at (3,0) {$S\times S^{1|1}$}; 
\node (C) at (0,-1.2) {$S^{1|1}_{\ell',\lambda'}$};
\node (D) at (3,-1.2) {$S\times S^{1|1}$};
\node (E) at (5,-.6) {$M$};
\draw[->] (A) to node [above] {$\simeq$} (B);
\draw[->] (A) to node [left] {$f$} (C);
\draw[->] (C) to node [above] {$\simeq$} (D);
\draw[->] (B) to node [above] {$\phi$} (E);
\draw[->,dashed] (D) to node [below] {$\phi'$} (E);
\path (0,-.6) coordinate (basepoint);
\end{tikzpicture}\label{eq:E11action}
\eeq
where the horizontal arrows are the pull back of the isomorphism in Lemma~\ref{lem:torusstnd} and the super Euclidean isometry $f$ is from Lemma~\ref{lem:supercircleE11} with $(\ell',\lambda')=(\ell\pm 2\eta\lambda,\pm \lambda)$. The arrow $\phi'$ is uniquely determined by these isomorphisms and the input map $\phi$. Hence, given $(\ell,\lambda,\phi)\in \R^{1|1}_{>0}(S)\times \Map(S^{1|1},M)(S)$ and an $S$-point of $\Euc_{1|1}$, the $\Euc_{1|1}$-action on $\mathcal{L}^{1|1}(M)$ outputs $(\ell',\lambda',\phi')$ as in~\eqref{eq:E11action}. 

\begin{rmk}\label{rmk:leftright}
Precomposition actions (such as the action of $\Euc_{1|1}$ on $\Map(S^{1|1},M)$ above) are most naturally right actions. Turning this into a left action involves inversion on the group: the formula for $\phi'$ in~\eqref{eq:E11action} involves $\phi$ and the \emph{inverse} of $f$. This inversion introduces signs in the formulas for the left $\Euc_{1|1}$-action on $\mathcal{L}^{1|1}_0(M)$ below. Our choice to work with left actions is consistent with Freed's conventions for classical supersymmetric field theories~\cite[pages~44-45]{5lectures}; see also~\cite[page~357]{strings1}. 
\end{rmk}

There is an evident $S^1$-action on $\mathcal{L}^{1|1}(M)$ coming from the precomposition action of $S^1=\E/\Z<\E^{1|1}/\Z$ on $\Map(S^{1|1},M)$. Since the quotient is given by~$S^{1|1}/S^1\simeq \R^{0|1}$, the $S^1$-fixed points are
\beq
\mathcal{L}^{1|1}_0(M):=\R^{1|1}_{>0}\times\Map(\R^{0|1},M)\subset \R^{1|1}_{>0}\times \Map(S^{1|1},M)=\mathcal{L}^{1|1}(M).\label{eq:constantloops}
\eeq
We identify an $S$-point of $\mathcal{L}^{1|1}_0(M)$ with a map $S^{1|1}_{\ell,\lambda}\to M$ that factors as
\beq
S^{1|1}_{\ell,\lambda}\simeq S\times S^{1|1}= S\times \R^{1|1}/\Z\stackrel{p}{\to} S\times \R^{0|1}\to M\label{eq:R01proj}
\eeq
where the map $p$ is induced by the projection $\R^{1|1}\to \R^{0|1}$. The action~\eqref{eq:E11action} preserves this factorization condition; we give an explicit formula in Lemma~\ref{lem:E11action} below. Hence, the inclusion~\eqref{eq:constantloops} is $\Euc_{1|1}$-equivariant. 

\begin{lem} \label{lem:11partiton}
There is a functor $\mathcal{L}_0^{1|1}(M)\to 1|1\EBord(M)$ that induces a restriction map 
\beq
{\rm restr}\colon 1|1\EFT(M)\to C^\infty(\mathcal{L}_0^{1|1}(M))^{\Euc_{1|1}}. \label{eq:claimed11}
\eeq
\end{lem}

\bp
The $1|1$-dimensional Euclidean bordism category over~$M$ is constructed by inputting the $1|1$-dimensional Euclidean geometry from Definition~\ref{defn:11superEuc} into the definition of a geometric bordism category~\cite[Definition~4.12]{ST11}. The result is a category~$1|1\EBord(M)$ internal to stacks on the site of supermanifolds; in particular $1|1\EBord(M)$ has a stack of morphisms consisting of proper families of $1|1$-dimensional Euclidean manifolds with a map to $M$, with additional decorations related to the source and target of a bordism. 

By Lemma~\ref{lem:supercircleEuc}, super Euclidean circles give examples of $S$-families of $1|1$-dimensional Euclidean manifolds. An $S$-point of $\mathcal{L}_0^{1|1}(M)$ therefore defines a proper $S$-family of $1|1$-Euclidean manifolds with a map to $M$ via~\eqref{eq:R01proj}. We can identify this with an $S$-family of morphisms in~$1|1\EBord(M)$ whose source and target are the empty supermanifold equipped with the unique map to $M$. This defines a functor $\mathcal{L}_0^{1|1}(M)\to 1|1\EBord(M)$ and a restriction map $1|1\EFT(M)\to C^\infty(\mathcal{L}_0^{1|1}(M))$. We refer to the discussion preceding~\cite[Definition~4.13]{ST11} for an explanation why the restriction to closed bordisms extracts a function from a field theory. 

Finally we argue that this restriction has image in $\Euc_{1|1}$-invariant functions. By definition, an isometry between $1|1$-dimensional Euclidean manifolds comes from the action of the super Euclidean group~$\Euc_{1|1}=\E^{1|1}\rtimes \Z/2$ on the open cover defining the super Euclidean manifold. By Lemma~\ref{lem:supercircleE11}, the action~\eqref{eq:E11action} on $\mathcal{L}_0^{1|1}(M)$ is therefore through super Euclidean isometries of super circles compatible with the maps to $M$. By definition, these isometries define isomorphisms between the bordisms~\eqref{eq:R01proj} in~$1|1\EBord(M)$. Functions on a stack are functions on objects invariant under the action of isomorphisms. Hence, the restriction $1|1\EFT(M)\to C^\infty(\mathcal{L}_0^{1|1}(M))$ necessarily takes values in functions invariant under $\Euc_{1|1}$, yielding the claimed map~\eqref{eq:claimed11}. 
\ep


\subsection{Computing the action of Euclidean isometries}

\begin{lem} \label{lem:E11action}
The left $\Euc_{1|1}$-action on $\R^{1|1}_{>0}\times \Map(\R^{0|1},M)$ is given by
\beq
(s,\eta,\pm 1)\cdot (\ell,\lambda,x,\psi)= \left(\ell\pm 2i\eta\lambda,\pm\lambda,x\pm (\frac{\lambda s}{\ell} -\eta)\psi, \pm e^{-i\frac{\eta\lambda}{\ell}} \psi) \right)\label{eq:lem:E11action}
\eeq
using notation for the functor of points
$$
(s,\eta,\pm 1)\in (\E^{1|1}\rtimes \Z/2)(S)\simeq \Euc_{1|1}(S),\quad (\ell,\lambda)\in \R^{1|1}_{>0}(S),$$$$ (x,\psi)\in \Pi TM(S)\simeq \Map(\R^{0|1},M)(S).
$$
\end{lem}

\bp Let $p_{\ell,\lambda}\colon S^{1|1}_{\ell,\lambda}\to S\times \R^{0|1}$ denote the composition of the left three maps in~\eqref{eq:R01proj}. Given $(s,\eta,\pm 1)\in \Euc_{1|1}(S)$, $(\ell,\lambda)\in \R^{1|1}_{>0}(S)$ and $(x,\psi)\in \Pi TM(S)\simeq \Map(\R^{0|1},M)(S)$, the goal of the lemma is to compute formulas for $(\ell',\lambda')\in \R^{1|1}_{>0}(S)$ and $(x',\psi')\in \Pi TM(S)$ in the diagram
\beq
\begin{tikzpicture}[baseline=(basepoint)];
\node (A) at (0,0) {$S^{1|1}_{\ell,\lambda}$};
\node (B) at (4,0) {$S\times \R^{0|1}$};
\node (C) at (0,-1.4) {$S^{1|1}_{\ell',\lambda'}$};
\node (D) at (4,-1.4) {$S\times \R^{0|1}$}; 
\node (E) at (8,-.6) {$M$};
\draw[->] (A) to node [above] {$p_{\ell,\lambda}$} (B);
\draw[->] (A) to node [left] {$f_{s,\eta,\pm 1}$} (C);
\draw[->] (C) to node [above] {$p_{\ell',\lambda'}$} (D);
\draw[->,dashed] (B) to (D);
\draw[->] (B) to node [above] {$(x,\psi)$} (E);
\draw[->] (D) to node [below] {$(x',\psi')$} (E);
\path (0,-.75) coordinate (basepoint);
\end{tikzpicture}\label{eq:R01map}
\eeq
where the arrow labeled by $f_{s,\eta,\pm 1}$ denotes the isometry between super Euclidean circles from Lemma~\ref{lem:supercircleE11} for $(s,\eta,\pm 1) \in \Euc_{1|1}(S)$. Hence, we see that $(\ell',\lambda')$ is given by~\eqref{eq:R11action}. To compute $(x',\psi')$, we find a formula for the dashed arrow in~\eqref{eq:R01map}. To start, consider the map
 $$
\tilde{p}_{\ell,\lambda} \colon \R^{1|1}\times S \to \R^{0|1}\times S,\quad \tilde{p}_{\ell,\lambda}(t,\theta)=\theta-\lambda\frac{t}{\ell},
$$
which is part of the inverse to the isomorphism~\eqref{eq:Zequiv}. Indeed, we check the $\Z$-invariance condition for the action~\eqref{eq:Zact},
$$
\tilde{p}_{\ell,\lambda} (n\cdot (t,\theta))=\tilde{p}_{\ell,\lambda} (n\ell+t+in\lambda\theta,n\lambda+\theta)=n\lambda+\theta-\lambda\frac{n\ell+t+in\lambda\theta}{\ell}=\theta-\lambda\frac{t}{\ell}.
$$
Hence $\tilde{p}_{\ell,\lambda}$ determines a map $p_{\ell,\lambda}\colon S^{1|1}_{\ell,\lambda}\to S\times \R^{0|1}$, which is the map in~\eqref{eq:R01map}. 
From this we see that the dashed arrow in~\eqref{eq:R01map} is unique and determined by 
$$
\theta \mapsto \pm \left(\theta+\eta-\lambda\frac{s+i\eta\theta}{\ell}\right),\quad (s,\eta,\pm 1)\in \Euc_{1|1}(S), \ \theta\in  \R^{0|1}(S). 
$$
The left action~\eqref{eq:E11action} is given by (see Remark~\ref{rmk:leftright} for an explanation of the signs)
$$
(x+\theta\psi)\mapsto x\pm \left(\theta-\eta-\lambda\frac{-s-i\eta\theta}{\ell}\right)\psi=x\pm(\frac{\lambda s}{\ell}-\eta)\psi\pm \theta(1-i\frac{\eta\lambda}{\ell})\psi
$$
which is the claimed formula for $(x',\psi')$.  
\ep

Just as $\R$-actions on ordinary manifolds are determined by flows of vector fields, $\E^{1|1}$-actions on supermanifolds are determined by the flow of an odd vector field. This comes from differentiating a left $\E^{1|1}$-action at zero and considering the action by the element $Q$ of the super Lie algebra, using the notation from~\eqref{eq:Wick}. Odd vector fields on supermanifolds are precisely odd derivations on their functions. We note the isomorphism 
\beq
C^\infty(\mathcal{L}^{1|1}_0(M))&\simeq& C^\infty(\R^{1|1}_{>0}\times \Map(\R^{0|1},M))\simeq C^\infty(\R^{1|1}_{>0})\otimes \Omega^\bullet(M)\nonumber\\
&\simeq& C^\infty(\R_{>0})[\lambda]\otimes \Omega^\bullet(M)\label{eq:fnonsuperloop}
\eeq
where (in an abuse of notation) let $\ell,\lambda\in C^\infty(\R^{1|1}_{>0})$ denote the coordinate functions associated with the universal family of super circles $S=\R^{1|1}_{>0}\to \R^{1|1}_{>0} \subset \R^{1|1}$. In the above, we used that $C^\infty(\Map(\R^{0|1},M))\simeq \Omega^\bullet(M)$ and that~$C^\infty(S\times T)\simeq C^\infty(S)\otimes C^\infty(T)$ for supermanifolds $S$ and $T$ using the projective tensor product of Fr\'echet algebras, e.g., see~\cite[Example~49]{HST}. Let $\deg\colon \Omega^\bullet(M)\to \Omega^\bullet(M)$ denote the (even) degree derivation determined by $\deg(\omega)=k\omega$ for $\omega\in \Omega^k(M)$. 

\begin{lem}\label{lem:E11generator}
The left $\E^{1|1}$-action~\eqref{eq:lem:E11action} on $\mathcal{L}^{1|1}_0(M)$ is generated by the odd derivation
\beq
\widehat{Q}:=2i \lambda \frac{d}{d\ell}\otimes {\rm id}-\id\otimes \dR-i\frac{\lambda}{\ell}\otimes {\rm deg}\label{eq:Q11}
\eeq
using the identification of functions~\eqref{eq:fnonsuperloop} where $\dR$ is the de~Rham differential and ${\rm deg}$ is the degree derivation on differential forms. \end{lem}

\begin{proof}
We recall that right invariant vector fields generate left actions, so that the infinitesimal action of $\E^{1|1}$ on $\mathcal{L}^{1|1}_0(M)$ is determined by the action of~$Q$. Furthermore, minus the de~Rham operator generates the left $\E^{0|1}$-action $(x,\psi)\mapsto (x-\eta\psi,\psi)$ on $\Pi TM$, and minus the degree derivation generates the left $\R^\times$-action $(x,\psi)\mapsto (x,u^{-1}\psi)$, e.g., see~\cite[\S3.4]{HKST}. Applying the derivation $Q=\partial_\eta+i\eta\partial_s$ to~\eqref{eq:lem:E11action} and evaluating at $(s,\eta)=0$ recovers~\eqref{eq:Q11}. 
\ep

\subsection{The proof of Proposition~\ref{prop:compute11}}\label{sec:prop11}

The $\Euc_{1|1}$-equivariant inclusion
$$
\R_{>0}\times \Map(\R^{0|1},M)\hookrightarrow \R^{1|1}_{>0}\times \Map(\R^{0|1},M)=\mathcal{L}^{1|1}_0(M)
$$ 
is along $S$-families of super circles with $\lambda=0$. So by Lemmas~\ref{lem:E11action} and~\ref{lem:E11generator} we have
$$
C^\infty(\R_{>0}\times \Map(\R^{0|1},M))^{\Euc_{1|1}}\simeq \Omega^\bullet(M;C^\infty(\R_{>0}))^{\E^{1|1}\rtimes \Z/2}\simeq \Omega^\ev_\cl(M;C^\infty(\R_{>0}))
$$
using~\eqref{eq:lem:E11action} to see that $\Z/2$ acts through the parity involution (so invariant functions are even forms) and~\eqref{eq:Q11} to see that the $\E^{1|1}$-action is generated by minus the de~Rham $\dR$ (so invariant functions are closed forms). This verifies the equality~\eqref{Eq:diffformdata} when $d=1$ and extracts the data~$Z$ from an element of $C^\infty(\mathcal{L}^{1|1}_0(M))^{\Euc_{1|1}}$. 

Next, observe that
\beq
C^\infty(\mathcal{L}^{1|1}_0(M))&\simeq& C^\infty(\R^{1|1}_{>0}\times \Map(\R^{0|1},M))\simeq \Omega^\bullet(M;C^\infty(\R^{1|1}_{>0}))\nonumber \\
&\simeq& \Omega^\bullet(M;C^\infty(\R_{>0})[\lambda]) \nonumber
\eeq
where the final isomorphism comes from Taylor expansion of functions on $\R^{1|1}_{>0}$ in the odd coordinate function $\lambda$. For convenience we choose the parameterization of functions
\beq
C^\infty(\mathcal{L}^{1|1}_0(M))\simeq \{\ell^{\deg/2}(Z+2i\lambda \ell^{1/2}Z_\ell)\mid Z,Z_\ell\in \Omega^\bullet(M;C^\infty(\R_{>0}))\},\label{eq:parameterize11}
\eeq
where $\ell^{\deg/2}\omega=\ell^{k/2}\omega$ for $\omega\in \Omega^k(M;C^\infty(\R_{>0}))$. We again have that $\Z/2<\Euc_{1|1}$ acts by the parity involution, so since $\lambda$ is odd and $\ell$ is even we find
$$
C^\infty(\mathcal{L}^{1|1}_0(M))^{\Z/2}=\left\{\ell^{\deg/2}(Z+2i\lambda \ell^{1/2}Z_\ell) \mid \begin{array}{l} Z\in \Omega^\ev(M;C^\infty(\R_{>0}))\\ Z_\ell\in \Omega^\odd(M;C^\infty(\R_{>0}))\end{array}\right\}.
$$
Next we compute
\beq
\widehat{Q}(\ell^{\deg/2}Z+2i\lambda \ell^{1/2}\ell^{\deg/2}Z_\ell)&=&2i\lambda\frac{d}{d\ell}(\ell^{\deg/2}Z)-\ell^{-1/2}\ell^{\deg/2}\dR Z-2i\lambda\ell^{\deg/2}\dR Z_\ell-i\frac{\lambda}{\ell}\ell^{\deg/2}\deg(Z)\nonumber\\
&=&-\ell^{-1/2}\ell^{\deg/2}\dR Z+2i\lambda\ell^{\deg/2}(\frac{dZ}{d\ell}-\dR Z_\ell)\nonumber
\eeq
where in the first equality we use that $\dR(\ell^{\deg/2}\omega)=\ell^{-1/2}\ell^{\deg/2}(\dR\omega)$, and in the second equality we expand $2i\lambda\frac{d}{d\ell}(\ell^{\deg/2}Z)$ using the product rule and then simplify. Hence
\beq
&&\widehat{Q}(\ell^{\deg/2}(Z+2i\lambda\ell^{1/2} Z_\ell))=0\quad \iff \quad \dR Z=0, \quad \dR Z_\ell=\frac{dZ}{d\ell}.\label{eq:Qclosed11}
\eeq
By Lemma~\ref{lem:E11generator}, $\widehat{Q}$ generates the $\E^{1|1}$-action and (since $\E^{1|1}$ is connected) $\widehat{Q}$-invariant functions are equivalent to $\E^{1|1}$-invariant functions. 
Finally, we identify even differential forms with elements of $\Omega^\bullet(M;C^\infty(\R_{>0})[\beta,\beta^{-1}])$ of total degree zero and odd differential forms with elements of $\Omega^\bullet(M;C^\infty(\R_{>0})[\beta,\beta^{-1}])$ of total degree~$-1$ (essentially replacing $\ell$ in~\eqref{eq:parameterize11} by $\beta$). This completes the proof of Proposition~\ref{prop:compute11}.

\subsection{Concordance classes of functions}\label{sec:refinecoccycle1}

For Proposition~\ref{mainprop} we require a refinement of the cocycle map.

\begin{defn}\label{defn:cocycle1}
Using the notation from Proposition~\ref{prop:compute11}, for each $\mu\in \R_{>0}$ define a map
$$
\widehat{\rm cocycle}_\mu\colon C^\infty(\mathcal{L}^{1|1}_0(M))^{\Euc_{1|1}}\to  (\Omega^\bullet_\cl(M;\C[\beta,\beta^{-1}]))^0,\qquad \widehat{\rm cocycle}_\mu(Z,Z_\ell)=Z(\mu),
$$
where $Z(\mu)$ denotes evaluation at $\mu\in \R_{>0}$ and $(\Omega^\bullet_\cl(M;\C[\beta,\beta^{-1}]))^0$ is the space of closed differential forms of total degree zero.  
\end{defn}

\begin{lem} \label{lem:independentofmu}
The composition
$$
C^\infty(\mathcal{L}^{1|1}_0(M))^{\Euc_{1|1}}\stackrel{\widehat{\rm cocycle}_\mu}{\longrightarrow}  (\Omega^\bullet_\cl(M;\C[\beta,\beta^{-1}]))^0\stackrel{{\rm de\ Rham}}{\longrightarrow} \H(M;\C[\beta,\beta^{-1}])
$$
agrees with~\eqref{eq:thm1} and hence is independent of $\mu$.
\end{lem}

\bp
The calculation~\eqref{eq:constantinell} shows 
$$
[\widehat{\rm cocycle}_\mu(Z,Z_\ell)]=[Z(\mu)]=[Z]={\rm cocycle}(Z,Z_\ell)\in \H(M;\C[\beta,\beta^{-1}])\subset \H(M;C^\infty(\R_{>0})[\beta,\beta^{-1}])
$$
In particular the class underlying $\widehat{\rm cocycle}_\mu(Z,Z_\ell)$ is independent of $\mu$. 
\ep

\begin{proof}[Proof of Proposition~\ref{mainprop}, $d=1$] 
Proposition~\ref{prop:compute11} implies that $M\mapsto C^\infty(\mathcal{L}^{1|1}_0(M))^{\Euc_{1|1}}$ is a sheaf on the site of smooth manifolds. The map in Definition~\ref{defn:cocycle1} is a morphism of sheaves
\beq
&&\widehat{\rm cocycle}_\mu\colon C^\infty(\mathcal{L}^{1|1}_0(-))^{\Euc_{1|1}}\to \Omega^\ev_\cl(-;\C[\beta,\beta^{-1}])\label{eq:takeconcordancethis}
\eeq
taking values in closed forms of even degree. By Stokes theorem, concordance classes of closed forms on a manifold $M$ are cohomology classes. Hence, taking concordance classes of the map~\eqref{eq:takeconcordancethis} applied to a manifold $M$ proves the proposition when $d=1$. 
\end{proof}

\subsection{The Chern character of a super connection}\label{sec:superconnection}

A \emph{super connection} $\A$ on a $\Z/2$-graded vector bundle $V\to M$ is an odd $\C$-linear map satisfying the Leibniz rule~\cite{Quillensuper} 
$$
\A\colon \Omega^\bullet(M;V)\to \Omega^\bullet(M;V), \quad  \A(fs)=\dR f\cdot s+(-1)^{|f|}f \A s,\quad f\in \Omega^\bullet(M), \ s\in \Omega^\bullet(M;V).
$$
One can express a super connection as a finite sum $\A=\sum_j \A_j$ where $\A_j\colon \Omega^\bullet(M;V)\to \Omega^{\bullet+j}(M;V)$ raises differential form degree by~$j$. Note that $\A_1$ is an ordinary connection on~$V$, and $\A_j$ is a differential form valued in $\End(V)^\odd$ if $j$ is even and $\End(V)^\ev$ if $j$ is odd. 
\emph{Super parallel transport} provides a functor, denoted ${\rm sPar}$, from the groupoid of $\Z/2$-graded vector bundles with super connection on $M$ to the groupoid of $1|1$-dimensional Euclidean field theories over~$M$
\beq
\begin{array}{ccccc} 
{\rm Vect}^\A(M) & \stackrel{{\rm sPar}}{\to}  & 1|1\EFT(M) & \stackrel{\rm res}{\to} & C^\infty(\R_{>0}\times \Map(\R^{0|1},M))^{\Euc_{1|1}}\\
(V,\A) & \mapsto & {\rm sPar}(V,\A) & \mapsto & {\rm sTr}(e^{\ell\A^2}).
\end{array}\label{eq:dumitmap}
\eeq
Part of this construction is given in~\cite{florin}, reviewed in~\cite[\S1.3]{ST11}. A different approach (satisfying stronger naturality properties required to construct the functor ${\rm sPar}$) is work in progress by Arnold~\cite{Bertram}. Evaluating the field theory ${\rm sPar}(V,\A)$ on closed bordisms determines the function ${\rm sTr}(e^{\ell\A^2})\in C^\infty(\R_{>0}\times \Map(\R^{0|1},M))$.
The parameterization~\eqref{eq:parameterize11} extracts the function~$Z$ determined by 
$$
\ell^{\deg/2}Z=\sTr(\exp(\ell\A^2)).
$$
Hence we find that $Z=\sTr(\exp(\A_\ell^2))$ for 
\beq
\A_\ell&=&\ell^{1/2}\A_0+\A_1+\ell^{-1/2}\A_2+\ell^{-1}\A_3+\dots\label{eq:famofsconn}
\eeq
The $\R_{>0}$-family of super connections~\eqref{eq:famofsconn} appears frequently in index theory, e.g., \cite{Quillensuper} and~\cite[\S9.1]{BGV}. By~\cite[Proposition~1.41]{BGV}, the failure for $Z$ to be independent of $\ell$ is measured by the exact form,
\beq
\frac{d}{d\ell}{\rm sTr}(e^{\A_\ell^2})=\dR\left(\sTr\big(\frac{d\A_\ell}{dt}  e^{\A_\ell^2}\big)\right).\label{eq:transgression}
\eeq
By Proposition~\ref{prop:compute11}, the data $Z={\rm sTr}(e^{\A_\ell^2})$ and $Z_\ell=\sTr(\frac{d\A_\ell}{dt}  e^{\A_\ell^2})$ determine an element of $C^\infty(\mathcal{L}^{1|1}_0(M))^{\Euc_{1|1}}$ refining the Chern character of the $\Z/2$-graded vector bundle~$V$. 

\begin{rmk}\label{rmk:Fei} If $\A=\nabla$ is an ordinary connection, the family~\eqref{eq:famofsconn} is independent of~$\ell$ and $Z_\ell=0$. This recovers Fei Han's~\cite{Han} identification of the Chern form ${\rm Tr}(\exp(\nabla^2))$ with dimensional reduction of the $1|1$-dimensional Euclidean field theory~${\rm sPar}(V,\nabla)$.
\end{rmk}


\section{A map from $2|1$-Euclidean field theories to complexified elliptic cohomology}
The main goal of this section is to prove Proposition~\ref{prop:compute21}. From the discussion in~\S\ref{sec:outline}, this proves Theorem~\ref{thm} when $d=2$. We also prove Proposition~\ref{mainprop} when $d=2$ and comment on connections with a de~Rham model for complex analytic elliptic cohomology, complexified TMF, and elliptic Euler classes. 


\subsection{The moduli space of super Euclidean tori}

We will use the two equivalent descriptions of $S$-points of $\R^{2|1}$
\beq
\R^{2|1}(S)&\simeq& \{x,y \in C^\infty(S)^\ev,\ \theta \in C^\infty(S)^\odd\mid (x)_{\rm red}=\overline{(x)}_{\rm red}, (y)_{\rm red}=\overline{(y)}_{\rm red}\}\label{eq:r211}\\
&\simeq& \{z,w \in C^\infty(S)^{\ev}, \theta\in C^\infty(S)^\odd\mid (z)_{\rm red}=\overline{(w)}_{\rm red}\},\label{eq:r212}
\eeq
where reality conditions are imposed on restriction of functions to the reduced manifold $S_{\rm red}\hookrightarrow S$. The isomorphism between~\eqref{eq:r211} and~\eqref{eq:r212} is $(x,y)\mapsto (x+iy,x-iy)=(z,w)$. Below we shall adopt the standard (though potentially misleading) notation $\overline{z}:=w$. We take similar notation for $S$-points of $\Spin(2)$, using the identification $\Spin(2)\simeq U(1)\subset \C$ with the unit complex numbers. This gives the description
\beq
&&\Spin(2)(S)\simeq U(1)(S)=\{\dil,\bar\dil\in C^\infty(S)^{\ev}\mid (\dil)_{\rm red}=\overline{(\bar\dil)}_{\rm red}, \  \dil \bar\dil=1 \}.\label{Eq:spin2}
\eeq

\begin{defn}
Let~$\E^{2|1}$ denote the super Lie group with underlying supermanifold $\R^{2|1}$ and multiplication
\beq
&&(z,\bar z,\theta)\cdot (z',\bar z',\theta')=(z+z',\bar z+\bar z'+\theta\theta',\theta+\theta'), \quad (z,\bar z,\theta),(z',\bar z',\theta')\in \R^{2|1}(S).\label{eq:E21mult}
\eeq
Define the \emph{super Euclidean group} as $\E^{2|1}\rtimes \Spin(2)$ where the semidirect product is defined by the action (using the notation~\eqref{Eq:spin2})
$$
(\dil,\bar\dil)\cdot (z,\bar z,\theta)=(\dil^2 z,\bar\dil^2 \bar z,\bar\dil \theta),\qquad (\dil,\bar\dil)\in \Spin(2)(S).
$$
\end{defn} 

The Lie algebra of $\E^{2|1}$ has one even generator and one odd generator. In terms of left invariant vector fields, these are $\partial_z$ and $D=\partial_\theta-\theta\partial_{\bar z}$, whereas in terms of right-invariant vector fields they are $\partial_z$ and $Q=\partial_\theta+\theta\partial_{\bar z}$. The super commutators are
\beq
[\partial_z,D]=0, \ [D,D]=-\partial_{\bar z}\quad {\rm and} \quad [\partial_z,Q]=0,\ [Q,Q]=\partial_{\bar z}. \label{eq:21Lie}
\eeq

Let $\Lat\subset \C\times \C$ denote the manifold of \emph{based lattices} in~$\C$ parameterizing pairs of nonzero complex numbers $\ell_1,\ell_2\in \C^\times$ such that $\ell_1/\ell_2\in \HH\subset \C$ is in the upper half plane. Equivalently, the pair $(\ell_1,\ell_2)$ generate a based oriented lattice in~$\C$. We observe that $(\ell_1,\ell_2)\mapsto (\ell_1,\ell_1/\ell_2)$ defines a diffeomorphism $\Lat\simeq \C^\times\times \HH$, so that $\Lat$ is indeed a manifold. When regarding $\Lat$ as a supermanifold, an $S$-point is specified by $(\ell_1,\bar\ell_1,\ell_2,\bar\ell_2)\in \Lat(S)\subset (\C\times \C)(S)$, following the notation from~\eqref{eq:r212}. 

\begin{defn}\label{defn:sLat} Define the generalized supermanifold of \emph{based (super) lattices} in $\R^{2|1}$ as the subfunctor $s\Lat\subset \R^{2|1}\times \R^{2|1}$ (viewing $\R^{2|1}\times \R^{2|1}$ as a representable presheaf) whose $S$-points are $(\ell_1,\bar\ell_1,\lambda_1),(\ell_2,\bar\ell_2,\lambda_2)\in \R^{2|1}(S)$ such that:
\begin{enumerate}
\item The pair commute for the multiplication~\eqref{eq:E21mult} on $\E^{2|1}(S)\simeq \R^{2|1}(S)$,
$$
(\ell_1,\bar\ell_1,\lambda_1)\cdot (\ell_2,\bar\ell_2,\lambda_2)=(\ell_2,\bar\ell_2,\lambda_2)\cdot(\ell_1,\bar\ell_1,\lambda_1)\in \E^{2|1}(S).
$$
\item The reduced map $S_\red\to (\R^{2|1}\times \R^{2|1})_{\rm red}\simeq \R^2\times \R^2\simeq \C\times \C$ determines a family of based oriented lattices in~$\C$, i.e., the image lies in $\Lat\subset \C\times \C$. 
\end{enumerate}
\end{defn}

\begin{rmk}\label{rmk:21subgroup} We observe that (1) is equivalent to requiring that $(\ell_1,\bar\ell_1,\lambda_1),(\ell_2,\bar\ell_2,\lambda_2)\in \E^{2|1}(S)$ generate a $\Z^2$-subgroup, i.e., a homomorphism $S\times \Z^2\to S\times \E^{2|1}$ over $S$. 
\end{rmk}

\begin{defn}\label{Defn:supertorus}
Given an $S$-point $\Lambda=((\ell_1,\bar\ell_1,\lambda_1),(\ell_2,\bar\ell_2,\lambda_2))\in s\Lat(S)$, define the family of $2|1$-dimensional \emph{super tori} as the quotient
\beq
T^{2|1}_{\Lambda}:=(S\times \R^{2|1})/\Z^2\label{eq:supertorus}
\eeq
for the free left $\Z^2$-action over $S$ determined by the 
formula
\beq
(n,m)\cdot (z,\bar z,\theta)&=& (z+n\ell_1+m\ell_2,\bar z+n(\bar\ell_1+\lambda_1\theta)+m(\bar\ell_2+\lambda_2\theta),n\lambda_1+m\lambda_2+\theta),\nonumber\\
&& (n,m)\in \Z^2(S), (z,\bar z,\theta)\in \R^{2|1}(S).\label{eq:Z2act}
\eeq
Equivalently this is the restriction of the left $\E^{2|1}$-action on $S\times \R^{2|1}$ to the $S$-family of subgroups $S\times \Z^2\subset S\times \E^{2|1}$ with generators over $S$ specified by $(\ell_1,\bar\ell_1,\lambda_1)$ and $(\ell_2,\bar\ell_2,\lambda_2)$. Define the \emph{standard super torus} as $T^{2|1}=\R^{2|1}/\Z^2$ for the quotient by the action for the standard inclusion $\Z^2\subset \R^2\subset \E^{2|1}$, i.e., for the square lattice. \end{defn}

\begin{rmk} The $S$-family of subgroups $S\times \Z^2\hookrightarrow S\times \E^{2|1}$ determined by $\Lambda$ (as in Remark~\ref{rmk:21subgroup}) is normal if and only if $\lambda_1=\lambda_2=0$. Hence, although the standard super torus~$T^{2|1}$ inherits a group structure from $\E^{2|1}$, generic super tori $T^{2|1}_{\Lambda}$ do not. 
\end{rmk}
\begin{rmk} There is a more general notion of a family of super tori where the action~\eqref{eq:Z2act} also incorporates pairs of elements in $\Spin(2)$. This moduli space has connected components corresponding to choices of spin structure on an ordinary torus, with the component from Definition~\ref{Defn:supertorus} corresponding to the odd (or periodic-periodic) spin structure. This turns out to be the relevant component of the moduli space to recover complex analytic elliptic cohomology. 
\end{rmk}


Stolz and Teichner's $(\mathbb{M},G)$-structures are discussed before Definition~\ref{defn:11superEuc}. 

\begin{defn}[\cite{ST11} \S4.2]\label{defn:21superEuc}
A \emph{super Euclidean} structure on a $2|1$-dimensional family $T\to S$ is an $(\mathbb{M},G)$-structure for the left action of $G=\E^{2|1}\rtimes \Spin(2)$ on~$\mathbb{M}=\R^{2|1}$. 
\end{defn}

\begin{lem} \label{lem:supertorusEuc}
An $S$-family of super tori~\eqref{eq:supertorus} has a canonical super Euclidean structure. 
\end{lem} 
\bp
The proof is the same as for Lemma~\ref{lem:supercircleEuc}, using the open cover $S\times \R^{2|1}\to T^{2|1}_\Lambda$ and transition data from the $\Z^2$-action~\eqref{eq:Z2act}. 
\ep

We observe that every family of super tori pulls back from the universal family $s\Lat\times \R^{2|1})/\Z^2\to s\Lat$ along a map $S\to s\Lat$. Hence, we regard
$$
\mathcal{M}^{2|1}:=s\Lat,\qquad \mathcal{T}^{2|1}:=(s\Lat\times \R^{2|1})/\Z^2\to s\Lat
$$
as the moduli space of super Euclidean tori and the universal family of super Euclidean tori, respectively. 
The following identifies $s\Lat$ with the moduli space of super Euclidean structures on the standard super torus.

\begin{lem} \label{lem:torusstndd} There exists an isomorphism of supermanifolds over~$s\Lat$,
\beq
&&s\Lat\times T^{2|1}\stackrel{\sim}{\to} \mathcal{T}^{2|1}, \label{eq:isotostd}
\eeq
from the constant $s\Lat$-family with fiber the standard super torus to the universal family of super Euclidean tori. This isomorphism does not preserve the super Euclidean structure on~$\mathcal{T}^{2|1}$.
\end{lem}
\bp
Define the map 
\beq
s\Lat \times \R^{2|1}&\to& s\Lat \times \R^{2|1}\label{eq:Z2equiv}\\
(\ell_1,\bar\ell_1,\lambda_1,\ell_2,\bar\ell_2,\lambda_2,x,y,\theta)&\mapsto& (\ell_1,\bar\ell_1,\lambda_1,\ell_2,\bar\ell_2,\lambda_2,\nonumber \\&& \ell_1 x+\ell_2 y,x(\bar \ell_1+\lambda_1\theta)+y(\bar \ell_2+\lambda_2\theta),\theta+x\lambda_1+y\lambda_2),\nonumber\\
&&(\ell_1,\bar\ell_1,\lambda_1,\ell_2,\bar\ell_2,\lambda_2) \in s\Lat (S),\  (x,y,\theta)\in \R^{2|1}(S).\nonumber
\eeq
where the source uses~\eqref{eq:r211} to specify an $S$-point $(x,y,\theta)\in \R^{2|1}(S)$ whereas the target uses~\eqref{eq:r212}. 
Observe that~\eqref{eq:Z2equiv} is $\Z^2$-equivariant for the actions on the source and target,
\beq
(n,m)\cdot (\ell_1,\bar\ell_1,\lambda_1,\ell_2,\bar\ell_2,\lambda_2,x,y,\theta)&=&(\ell_1,\bar\ell_1,\lambda_1,\ell_2,\bar\ell_2,\lambda_2,x+n,y+m,\theta), \nonumber\\
(n,m)\cdot (\ell_1,\bar\ell_1,\lambda_1,\ell_2,\bar\ell_2,\lambda_2,z,\bar z,\theta)&=&(\ell_1,\bar\ell_1,\lambda_1,\ell_2,\bar\ell_2,\lambda_2,z+n\ell_1+m\ell_2,\nonumber\\
&&\bar z+n(\bar\ell_1+\lambda_1\theta)+m(\bar\ell+\lambda_2\theta),\theta+n\lambda+m\lambda_2),\nonumber
\eeq
respectively. Hence~\eqref{eq:Z2equiv} determines a map between the respective $\Z^2$-quotients, defining a map~\eqref{eq:isotostd}. This map is easily seen to be an isomorphism of supermanifolds. Since the map~\eqref{eq:isotostd} is not locally determined by the action of $\E^{2|1}\rtimes \Spin(2)$ on $\R^{2|1}$, it is not a super Euclidean isometry. 
\ep

\begin{defn}
Define the super Lie group $\Euc_{2|1}:=\E^{2|1}\rtimes \Spin(2)\times \SL_2(\Z)$. 
\end{defn}

The following gives an $S$-point formula for the action of $\Euc_{2|1}$ on~$\mathcal{T}^{2|1}$ and $\mathcal{M}^{2|1}=s\Lat$ coming from isometries between super Euclidean tori. 


\begin{lem}\label{lem:torusE21} Given $\Lambda=((\ell_1,\bar\ell_1,\lambda_1),(\ell_2,\bar\ell_2,\lambda_2))\in s\Lat(S)=\mathcal{M}^{2|1}(S)$, $(w,\bar w,\eta,\dil,\bar\dil)\in (\E^{2|1}\rtimes \Spin(2))(S)$, and $\gamma\in \SL_2(\Z)(S)$, there is an isomorphism $f_{(w,\bar w,\eta,\dil,\bar\dil)}\colon T^{2|1}_{\Lambda}\to T^{2|1}_{\Lambda'}$ of super Euclidean tori over $S$ sitting in the diagram
\beq
\begin{tikzpicture}[baseline=(basepoint)];
\node (A) at (0,0) {$S\times \R^{2|1}$};
\node (B) at (5,0) {$S\times \R^{2|1}$}; 
\node (C) at (0,-1.3) {$T^{2|1}_{\Lambda}$};
\node (D) at (5,-1.3) {$T^{2|1}_{\Lambda'}$};
\draw[->] (A) to node [above] {$(w,\bar w,\eta,\dil,\bar\dil)\cdot $} (B);
\draw[->] (A) to (C);
\draw[->] (C) to node [above] {$f_{(w,\bar w,\eta,\dil,\bar\dil)}$} (D);
\draw[->] (B) to  (D);
\path (0,-.75) coordinate (basepoint);
\end{tikzpicture}\nonumber
\eeq
where the upper horizontal arrow is determined by the left $\E^{2|1}\rtimes \Spin(2)$-action on $\R^{2|1}$, the left vertical arrow is the quotient map~\eqref{eq:supercircle} for $\Lambda$, and the right vertical arrow is the quotient map for
\beq\label{eq:R21action}
\Lambda'&:=&
\left(\begin{array}{c} \big(\dil^2(a\ell_1+b\ell_2),\bar{\dil}^2(a(\bar{\ell}_1+2\eta\lambda_1)+b(\bar{\ell}_2+2\eta\lambda_2)),\bar{\dil}(a\lambda_1+b\lambda_2)\big), \\ 
\big(\dil^2(c\ell_1+d\ell_2),\bar{\dil}^2(c(\bar{\ell}_1+2\eta\lambda_1)+d(\bar{\ell}_2+2\eta\lambda_2)),\bar{\dil}(c\lambda_1+d\lambda_2)\big)\end{array}\right)
\eeq
where $\gamma=\left[\begin{array}{cc} a & b \\ c & d\end{array}\right]\in \SL_2(\Z)(S)$. 
\end{lem}
%

\bp
Consider the diagram
\beq
\begin{tikzpicture}[baseline=(basepoint)];
\node (A) at (0,0) {$\Z^2\times S\times \R^{2|1}$};
\node (B) at (5,0) {$\Z^2\times S\times \R^{2|1}$};
\node (C) at (0,-1.3) {$S\times \R^{2|1}$};
\node (D) at (5,-1.3) {$S\times \R^{2|1}.$}; 
\draw[->] (A) to node [above] {$\gamma\times (w,\bar w,\eta,\dil,\bar\dil)$} (B);
\draw[->] (A) to node [left] {$\Lambda$} (C);
\draw[->] (C) to node [below] {$(w,\bar w,\eta,\dil,\bar\dil)$} (D);
\draw[->] (B) to node [right] {$\Lambda'$} (D);
\path (0,-.75) coordinate (basepoint);
\end{tikzpicture}\label{eq:thediagram1234}
\eeq
The horizontal arrows are determined by the left action of $(w,\bar w,\eta)\in \E^{2|1}(S)$, $(\dil,\bar\dil)\in \Spin(2)(S)$ on $S\times \R^{2|1}$ and a map $S\times \Z^2\to S\times \Z^2$ specified by $\gamma\in \SL_2(\Z)(S)$. The vertical arrows are the $\Z^2$-action on $S\times \R^{2|1}$ generated by $\Lambda,\Lambda'\in s\Lat(S)$. Using~\eqref{eq:E21mult}, this square commutes if and only if~\eqref{eq:R21action} holds. Commutativity of~\eqref{eq:thediagram1234} gives a map on the $\Z^2$-quotients, which is precisely a map $T^{2|1}_\Lambda\to T^{2|1}_{\Lambda'}$. This map is locally given by the action of $\E^{2|1}\rtimes \Spin(2)$ on $\R^{2|1}$, so by construction it respects the super Euclidean structures. 
\ep

We will require an explicit description of functions on $s\Lat$, i.e., the morphisms of presheaves $s\Lat\to C^\infty$. 
Regarding $\Lat$ as a representable presheaf on supermanifolds, there is an evident monomorphism $\Lat\hookrightarrow s\Lat$ from the canonical inclusion $\C\times \C\simeq \R^2\times \R^2\hookrightarrow \R^{2|1}\times \R^{2|1}$. In the following, let $\lambda_1,\lambda_2\in C^\infty(s\Lat)$ denote the restriction of the odd coordinate functions $C^\infty(\R^{2|1}\times \R^{2|1})\simeq C^\infty(\R^4)[\lambda_1,\lambda_2]$ under the inclusion $s\Lat\subset \R^{2|1}\times \R^{2|1}$. 

\begin{lem} \label{lem:sLatfun}
There is an isomorphism of algebras~$C^\infty(s\Lat)\simeq C^\infty(\Lat)[\lambda_1,\lambda_2]/(\lambda_1\lambda_2)$.
\end{lem}

\bp
Consider the composition
$$
S\to s\Lat\subset \R^{2|1}\times \R^{2|1}\stackrel{p_1,p_2}{\longrightarrow} \R^{2|1}
$$
where, as usual, we write the associated pair of maps $S\to \R^{2|1}$ as $(\ell_1,\bar\ell_1,\lambda_1)$ and $(\ell_2,\bar\ell_2,\lambda_2)$. We therefore have 4 even and 2 odd functions on $s\Lat$ that (as maps of sheaves $s\Lat\to C^\infty$) assign to an $S$-point the functions $\ell_1,\bar\ell_1,\ell_2,\bar\ell_2\in C^\infty(S)^\ev$ or $\lambda_1,\lambda_2\in C^\infty(S)^\odd$. It is easy to see that arbitrary smooth functions in the variables $\ell_1,\bar\ell_1,\ell_2,\bar\ell_2$ continue to define maps of sheaves and hence smooth functions on $s\Lat$. Furthermore, since these are the restriction of functions on $\R^2\times \R^2\subset \R^{2|1}\times \R^{2|1}$, we can identify them with functions on~$\Lat$. This specifies the even subalgebra $C^\infty(\Lat)\subset C^\infty(s\Lat)$. On the other hand, the odd functions $\lambda_1$ and $\lambda_2$ are subject to a relation coming from condition (1) in Definition~\ref{defn:sLat}, namely that $\lambda_1\lambda_2=\lambda_2\lambda_1\in C^\infty(S)^\odd$ for all $S$. Since these are odd functions, this is equivalent to the condition that $\lambda_1\lambda_2=0$. 
Hence the functions on $s\Lat$ are as claimed. 
\ep

\begin{rmk} 
The relation $\lambda_1\lambda_2=0$ implies that $C^\infty(s\Lat)$ is not the algebra of functions on any supermanifold, and hence the generalized supermanifold $s\Lat$ fails to be representable. 
\end{rmk}
\subsection{Super Euclidean double loop spaces}

\begin{defn} Define the \emph{super Euclidean double loop space} as the generalized supermanifold
$$
\mathcal{L}^{2|1}(M):=s\Lat\times \Map(T^{2|1},M).
$$ 
We identify an $S$-point of $\mathcal{L}^{2|1}(M)$ with a map $T^{2|1}_{\Lambda}\to M$ given by the composition
\beq
T^{2|1}_{\Lambda}\simeq  S\times T^{2|1}\to M,\label{eq:torusloop}
\eeq
using the isomorphism from Lemma~\ref{lem:torusstndd}. 
\end{defn} 

We shall define a left action of $\Euc_{2|1}$ on $\mathcal{L}^{2|1}(M)$ determined by the diagram
\beq
\begin{tikzpicture}[baseline=(basepoint)];
\node (A) at (0,0) {$T^{2|1}_\Lambda$};
\node (B) at (3,0) {$S\times T^{2|1}$}; 
\node (C) at (0,-1.2) {$T^{2|1}_{\Lambda'}$};
\node (D) at (3,-1.2) {$S\times T^{2|1}$};
\node (E) at (5,-.6) {$M$};
\draw[->] (A) to node [above] {$\simeq$} (B);
\draw[->] (A) to node [left] {$f$} (C);
\draw[->] (C) to node [above] {$\simeq$} (D);
\draw[->] (B) to node [above] {$\phi$} (E);
\draw[->,dashed] (D) to node [below] {$\phi'$} (E);
\path (0,-.6) coordinate (basepoint);
\end{tikzpicture}\label{eq:E21action}
\eeq
where the horizontal arrows are the inverses of the isomorphisms of supermanifolds pulled back from Lemma~\ref{lem:torusstndd}, and $f$ is the super Euclidean isometry associated to an $S$-point of $\Euc_{2|1}$ in Lemma~\ref{lem:torusE21}. These isomorphisms together with the arrow $\phi$ uniquely determine~$\phi'$ in~\eqref{eq:E21action}. Hence, for $(\Lambda,\phi)\in s\Lat(S)\times \Map(T^{2|1},M)(S)$ and an $S$-point of $\Euc_{2|1}$, we define the $\Euc_{2|1}$-action on $\mathcal{L}^{2|1}(M)$ as outputing $(\Lambda',\phi')$ in~\eqref{eq:E21action}. We caution that this is a left $\Euc_{2|1}$-action on $s\Lat\times \Map(T^{2|1},M)$, and refer to Remark~\ref{rmk:leftright} for a discussion of left actions on mapping spaces. 

There is an evident $T^2$-action on $\mathcal{L}^{2|1}(M)$ coming from the $T^2$-action on $\Map(T^{2|1},M)$ by the precomposition action of $T^2$ on $T^{2|1}$. The $T^2$-fixed points comprise the subspace
\beq
&&\mathcal{L}^{2|1}_0(M):=s\Lat\times\Map(\R^{0|1},M)\subset s\Lat\times \Map(T^{2|1},M)=\mathcal{L}^{2|1}(M).\label{eq:constantmaps}
\eeq
We identify an $S$-point of this subspace as a map $T^{2|1}_{\Lambda}\to M$ that factors as
\beq
T^{2|1}_{\Lambda}\simeq S\times T^{2|1}\simeq S\times \R^{2|1}/\Z^2\stackrel{p}{\to} S\times \R^{0|1}\to M,\label{eq:R01proj2}
\eeq
where the map $p$ is induced by the projection $\R^{2|1}\to \R^{0|1}$. The action~\eqref{eq:E21action} preserves this factorization condition; we give explicit formulae in Lemma~\ref{lem:E21action} below. Hence, the inclusion~\eqref{eq:constantmaps} is $\Euc_{2|1}$-equivariant.

\begin{lem} \label{lem:21partiton}
There is a functor $\mathcal{L}_0^{2|1}(M)\to 2|1\EBord(M)$ that induces a restriction map 
\beq
{\rm restr}\colon 2|1\EFT(M)\to C^\infty(\mathcal{L}_0^{2|1}(M))^{\Euc_{2|1}}. \label{eq:claimed21}
\eeq
\end{lem}

\bp
The proof is completely analogous to that of Lemma~\ref{lem:11partiton}. Namely, Lemma~\ref{lem:supertorusEuc} gives a functor $\mathcal{L}^{2|1}_0(M)\to 2|1\Bord(M)$, and Lemma~\ref{lem:torusE21} shows that the action of $\Euc_{2|1}$ on $\mathcal{L}_0^{2|1}(M)$ is through isomorphisms between $S$-families of $2|1$-dimensional Euclidean bordisms. Hence, the restriction map lands in $\Euc_{2|1}$-invariant functions. 
\ep

\subsection{Computing the action of super Euclidean isometries}\label{eq:computeaction}
\begin{defn}
Using the notation from Lemma~\ref{lem:sLatfun}, define the function 
\beq\label{eq:voldef}
\vol:=\frac{\ell_1\bar\ell_2-\bar\ell_1\ell_2}{2i}\in C^\infty(s\Lat). 
\eeq
\end{defn} 

The restriction of $\vol$ along $\Lat\hookrightarrow s\Lat$ is the function that reads off the volume of an ordinary torus $\C/\ell_1\Z\oplus\ell_2\Z$ using the flat metric. In particular, this function is real-valued, positive, and invertible. By Lemma~\ref{lem:sLatfun}, the function $\vol$ on $s\Lat$ is also invertible.

\begin{lem} \label{lem:E21action}
The left $\E^{2|1}\rtimes \Spin(2)$-action on $s\Lat\times \Map(\R^{0|1},M)$ is given by 
\beq
(w,\bar w,\eta,\dil,\bar\dil)\cdot (\ell_1,\bar\ell_1,\lambda_1,\ell_2,\bar\ell_2,\lambda_2,x,\psi)&=&\Big(\dil^2\ell_1,\bar\dil^2(\bar\ell_1+2\eta\lambda_1),\bar\dil\lambda_1,\dil^2\ell_2,\bar\dil^2(\bar\ell_2+2\eta\lambda_2),\bar\dil\lambda_2,\nonumber \\
&&x-\bar\dil^{-1}\left(\eta+\frac{\lambda_1\ell_2-\lambda_2\ell_1}{2i\vol}\bar w+\frac{\lambda_1\bar\ell_2-\lambda_2\bar\ell_1}{2i\vol}w\right)\psi,\label{eq:lem:E21action}\\
&&\bar\dil^{-1}\exp\left(\eta\frac{\lambda_1\ell_2-\lambda_2\ell_1}{2i\vol}\right)\psi\Big)\nonumber
\eeq
where 
$$
(w,\bar w,\eta)\in \E^{2|1}(S),\ (\dil,\bar\dil)\in \Spin(2)(S),\  (x,\psi)\in \Pi TM(S)\simeq \Map(\R^{0|1},M)(S).
$$
The $\SL_2(\Z)$-action on $s\Lat\times \Map(\R^{0|1},M)$ is diagonal for the action on $s\Lat$ from~\eqref{eq:R21action} and the trivial action on $\Map(\R^{0|1},M)$. 
\end{lem}

\bp Let $p_{\Lambda}\colon T^{2|1}_{\Lambda}\to S\times \R^{0|1}$ denote the composition of the left three maps in~\eqref{eq:R01proj2}. Given $(w,\bar w,\eta)\in \E^{2|1}(S)$, $(\dil,\bar\dil)\in \Spin(2)(S)$, $\Lambda\in s\Lat(S)$ and $(x,\psi)\in \Pi TM(S)$, the goal of the lemma is to compute formulas for $\Lambda'\in s\Lat(S)$ and $(x',\psi')\in \Pi TM(S)$ in the diagram
\beq
\begin{tikzpicture}[baseline=(basepoint)];
\node (A) at (0,0) {$T^{2|1}_{\Lambda}$};
\node (B) at (4,0) {$S\times \R^{0|1}$};
\node (C) at (0,-1.4) {$T^{2|1}_{\Lambda'}$};
\node (D) at (4,-1.4) {$S\times \R^{0|1}$}; 
\node (E) at (8,-.6) {$M$};
\draw[->] (A) to node [above] {$p_{\Lambda}$} (B);
\draw[->] (A) to node [left] {$f_{(w,\bar w,\eta,\dil,\bar\dil)}$} (C);
\draw[->] (C) to node [above] {$p_{\Lambda'}$} (D);
\draw[->,dashed] (B) to (D);
\draw[->] (B) to node [above] {$(x,\psi)$} (E);
\draw[->] (D) to node [below] {$(x',\psi')$} (E);
\path (0,-.75) coordinate (basepoint);
\end{tikzpicture}\label{eq:R01map2}
\eeq
where the arrow labeled by $f_{(w,\bar w,\eta,\dil,\bar\dil)}$ denotes the associated map between super Euclidean tori from Lemma~\ref{lem:torusE21}. For the first statement in the present lemma we take $\gamma=\id\in \SL_2(\Z)(S)$. We see that $\Lambda'$ is given by~\eqref{eq:R21action}. To compute $(x',\psi')$, we find a formula for the dashed arrow in~\eqref{eq:R01map2} that makes the triangle commute. To start, part of the data of the inverse to the isomorphism~\eqref{eq:Z2equiv} is
\beq
&&\tilde{p}_{\Lambda} \colon S\times \R^{2|1} \to S\times \R^{0|1},\quad \tilde{p}_{\Lambda}(z,\bar z,\theta)=\theta-\lambda_1\frac{\bar{z}\ell_2-z\bar{\ell}_2}{2i\vol}-\lambda_2\frac{z\bar{\ell}_1-\bar{z}\ell_1}{2i\vol}.\label{eq:Ptild}
\eeq
We verify that $\tilde{p}_\Lambda$ is $\Z^2$-invariant for the action~\eqref{eq:Z2act},
\beq
\resizebox{\textwidth}{!}{$\begin{array}{lll}
\tilde{p}_{\ell,\lambda} ((n,m)\cdot (z,\bar z,\theta))&=&\tilde{p}_{\ell,\lambda} (z+n\ell_1+m\ell_2,\bar z+n(\bar\ell_1+\lambda_1\theta)+m(\bar\ell_2+\lambda_2\theta),n\lambda_1+m\lambda_2+\theta)\nonumber \\
&=&n\lambda_1+m\lambda_2+\theta-\lambda_1\frac{(\bar z+n(\bar\ell_1+\lambda_1\theta)+m(\bar\ell_2+\lambda_2\theta))\ell_2- (z+n\ell_1+m\ell_2)\bar\ell_2}{2i \vol}\nonumber\\
&&-\lambda_2\frac{(z+n\ell_1+m\ell_2)\bar\ell_1- (\bar z+n(\bar\ell_1+\lambda_1\theta)+m(\bar\ell_2+\lambda_2\theta))\ell_1}{2i \vol}\nonumber \\
&=&\theta-\lambda_1\frac{\bar{z}\ell_2-z\bar{\ell}_2}{2i\vol}-\lambda_2\frac{z\bar{\ell}_1-\bar{z}\ell_1}{2i\vol},\nonumber\end{array}$}
\eeq
where we used~\eqref{eq:voldef}. 
Hence $\tilde{p}_{\Lambda}$ determines a map $p_{\Lambda}\colon T^{2|1}_{\Lambda}\to S\times \R^{0|1}$, which is the map in~\eqref{eq:R01map2}. 
From this we see that the dashed arrow in~\eqref{eq:R01map} is unique and determined by 
\beq
\theta\mapsto \bar{\dil}\left(\theta+\eta-\frac{(\lambda_1\ell_2-\lambda_2\ell_1)(\bar{w}+\theta\eta)-(\lambda_1\bar{\ell}_2-\lambda_2\bar{\ell}_1)w}{2i\vol}\right).\label{eq:21action}
\eeq
Following Remark~\ref{rmk:leftright}, the left action of $\E^{2|1}\rtimes \Spin(2)$ on $(x+\theta\psi)\in \Map(\R^{0|1},M)(S)$ is given by
\beq
(x+\theta\psi)&\mapsto& x +\bar\dil^{-1}\left(\theta-\eta-\frac{(\lambda_1\ell_2-\lambda_2\ell_1)(-\bar{w}-\theta\eta)+(\lambda_1\bar{\ell}_2-\lambda_2\bar{\ell}_1)w}{2i\vol}\right)\psi\nonumber \\
&=&x-\bar\dil^{-1}\left(\eta+\frac{(\lambda_1\ell_2-\lambda_2\ell_1)\bar{w}+(\lambda_1\bar\ell_1-\lambda_2\bar\ell_1)w}{2i \vol}\right)\psi+\bar\dil^{-1}\theta\left(1+\eta \frac{\lambda_1\ell_2-\lambda_2\ell_1}{2i \vol}\right)\psi\nonumber
\eeq
which gives the claimed formula for $(x',\psi')$. Finally, a short computation shows that $p_{\Lambda}=p_{\Lambda'}\circ \gamma$ where $\gamma\colon T^{2|1}_\Lambda\to T^{2|1}_{\Lambda'}$ is the isometry associated to $\gamma\in \SL_2(\Z)(S)$ from Lemma~\ref{lem:torusE21}. Hence, the $\SL_2(\Z)$-action on $s\Lat\times \Map(\R^{0|1},M)$ is indeed through the action on $s\Lat$.
\ep

From the Lie algebra description~\eqref{eq:21Lie}, a left $\E^{2|1}$-action determines an even and an odd vector field gotten by considering the infinitesimal action by the elements $Q=\partial_\theta+\theta\partial_{\bar z}$ and $\partial_z$ of the Lie algebra of $\E^{2|1}$. We note the isomorphisms
\beq
C^\infty(\mathcal{L}^{1|1}_0(M))&\simeq& C^\infty(s\Lat\times \Map(\R^{0|1},M))\simeq C^\infty(\Lat\times \Map(\R^{0|1},M))[\lambda_1,\lambda_2]/(\lambda_1\lambda_2)\nonumber\\
&\simeq& \big(C^\infty(\Lat)\otimes \Omega^\bullet(M)\big)[\lambda_1,\lambda_2]/(\lambda_1\lambda_2)\label{eq:fnontori}
\eeq
where in~\eqref{eq:fnontori} we used that the projective tensor product of Fr\'echet spaces satisfies $C^\infty(S\times T)\simeq C^\infty(S)\otimes C^\infty(T)$ for supermanifolds $S$ and $T$, e.g., see~\cite[Example~49]{HST}.

\begin{lem}\label{lem:infaction21} The derivative at 0 of the left $\E^{2|1}$-action on $\mathcal{L}^{2|1}_0(M)$ from~\eqref{eq:lem:E21action} is determined by the derivations on~$C^\infty(\mathcal{L}^{2|1}_0(M))$
\beq
&&\widehat{\partial}_w=\frac{\lambda_1\bar\ell_2-\lambda_2\bar\ell_1}{2i\vol}\otimes \dR,\quad \widehat{Q}=2\lambda_1\partial_{\bar\ell_1}\otimes \id+2\lambda_2\partial_{\bar\ell_2}\otimes \id-\id\otimes \dR-\frac{\lambda_2\ell_1-\lambda_1\ell_2}{2i\vol}\otimes {\rm deg},\label{eq:Q21}
\eeq
where $\dR$ is the de~Rham differential and ${\rm deg}$ is the degree endomorphism on forms.

\end{lem}
\begin{proof} The proof follows the same reasoning as the proof of Lemma~\ref{lem:E11generator}, using that right invariant vector fields generate left actions and that the $\E^{0|1}\rtimes \C^\times$-action on $\Map(\R^{0|1},M)$ is generated by minus the de~Rham operator and the degree derivation. In this case we apply the derivation $Q=\partial_\eta+\eta\partial_{\bar w}$ and $\partial_w$ to~\eqref{eq:lem:E21action} (with $(\dil,\bar\dil)=(1,1)$) and evaluate at $(w,\bar w,\eta)=(0,0,0)$ to obtain~\eqref{eq:Q21}. 
\ep

\subsection{The proof of Proposition~\ref{prop:compute21}}

Functions on $\mathcal{L}^{2|1}_0(M)$ can be described as
\beq
C^\infty(\mathcal{L}^{2|1}_0(M))&= &C^\infty(s\Lat\times \Map(\R^{0|1},M))\nonumber\\
&\simeq& \Omega^\bullet(M;C^\infty(s\Lat))\simeq \Omega^\bullet(M;C^\infty(\Lat)[\lambda_1,\lambda_2]/(\lambda_1\lambda_2))\label{eq:itsadescription}\\
&\simeq&\Omega^\bullet(M;C^\infty(\Lat))\oplus \lambda_1\cdot \Omega^\bullet(M;C^\infty(\Lat))\oplus\lambda_2\cdot \Omega^\bullet(M;C^\infty(\Lat)), \nonumber
\eeq
using Lemma~\ref{lem:sLatfun} in the 2nd line, and where the isomorphism in the 3rd line is additive. We start by proving a version of Proposition~\ref{prop:compute21} for invariants by $\E^{2|1}\rtimes \Z/2<\E^{2|1}\rtimes \Spin(2)\times \SL_2(\Z)=\Euc_{2|1}$. Analogously to the notation in~\S\ref{sec:prop11}, let $\vol^{\deg}\omega=\vol^k\omega$ for~$\omega\in \Omega^k(M)$. 

\begin{lem}\label{lem:spininvt}
Any element $\omega\in C^\infty(\mathcal{L}^{2|1}_0(M))^{\Z/2}$ can be written as
\beq
&&\omega=\vol^{\deg/2}(\omega_0+2\lambda_1\vol^{1/2}\omega_1+2\lambda_2\vol^{1/2}\omega_2)\label{eq:spininvt}
\eeq
where $\omega_0\in \Omega^\ev(M;C^\infty(\Lat))$ and $\omega_1,\omega_2\in \Omega^\odd(M;C^\infty(\Lat))$. A $\Z/2$-invariant function~$\omega$ expressed as~\eqref{eq:spininvt} is $\E^{2|1}$-invariant if and only if 
\beq
\dR\omega_0=0,\quad \partial_{\bar \ell_1}\omega_0=\dR\omega_1,\quad \partial_{\bar \ell_2}\omega_0=\dR\omega_2\label{eq:E21}
\eeq
where $\dR$ is the de~Rham differential on $M$. 
\end{lem}

\bp
The element~$-1\in U(1)\simeq \Spin(2)$ acts through the parity involution, which on $C^\infty(s\Lat)$ is determined by~$\lambda_i\mapsto -\lambda_i$. Using~\eqref{eq:itsadescription} and the fact that $\vol$ is an invertible function on~$\Lat$, we see that any $\Z/2$-invariant function can be written in the form~\eqref{eq:spininvt}. Next we compute for $\omega_0\in \Omega^k(M;C^\infty(\Lat))$
\beq
2(\lambda_1\partial_{\bar \ell_1}+\lambda_2\partial_{\bar \ell_2})(\vol^{k/2}\omega_0)&=&2(\lambda_1\partial_{\bar \ell_1}+\lambda_2\partial_{\bar \ell_2})\left(\left(\frac{\ell_1\bar\ell_2-\bar\ell_1\ell_2}{2i}\right)^{k/2}\omega_0\right)\nonumber\\
&=&\frac{\lambda_2\ell_1-\lambda_1\ell_2}{2i\vol}\deg(\vol^{k/2}\omega_0)+ 2\vol^{k/2}(\lambda_1\partial_{\bar \ell_1}+\lambda_2\partial_{\bar \ell_2})\omega_0.\nonumber
\eeq
So by Lemma~\ref{lem:infaction21}
$$
\widehat{Q}(\vol^{\deg/2}\omega_0)=\vol^{\deg/2}\left(2(\lambda_1\partial_{\bar \ell_1}+\lambda_2\partial_{\bar \ell_2})\omega_0-\vol^{-1/2}\dR\omega_0\right).
$$
Using that $\lambda_1^2=\lambda_2^2=\lambda_1\lambda_2=0$, we compute 
\beq
&&\widehat{Q}(\vol^{\deg/2}\omega_0+2\lambda_1\vol^{(\deg+1)/2} \omega_1+2\lambda_2\vol^{(\deg+1)/2}\omega_2)\nonumber\\
&&=\vol^{\deg/2}\big(2(\lambda_1\partial_{\bar \ell_1}+\lambda_2\partial_{\bar \ell_2})\omega_0-\vol^{-1/2}\dR\omega_0-2\lambda_1 \dR\omega_1-2\lambda_2 \dR\omega_2\big)\nonumber.
\eeq
Matching coefficients of $\lambda_1,\lambda_2$, the condition $\widehat{Q}\omega=0$ is therefore equivalent to~\eqref{eq:E21}.
Finally, invariance under the operator $\widehat{\partial}_w$ from Lemma~\ref{lem:infaction21} follows from being $\widehat{Q}$-closed, specifically from $\dR \omega_0=0$. Since $\E^{2|1}$ is connected with Lie algebra generated by~$\widehat{Q}$ and~$\widehat{\partial}_w$, we find that~\eqref{eq:E21} completely specifies the subalgebra $C^\infty(\mathcal{L}_0^{2|1}(M))^{\E^{2|1}\rtimes \Z/2}\subset C^\infty(\mathcal{L}^{2|1}_0(M))^{\Z/2}$. 
\ep

Next we compute the $\Spin(2)$-invariant functions. Consider the surjective map
\beq\label{eq:HRprofmap}
\varphi\colon\Lat\to \HH\times\R_{>0}\qquad (\ell_1,\bar\ell_1,\ell_2,\bar\ell_2)\mapsto (\ell_1/\ell_2,\bar\ell_1/\bar\ell_2,\vol)\in (\HH\times \R_{>0})(S) 
\eeq
and use the pullback on functions to get an injection
\beq
 C^\infty(\HH\times\R_{>0})[\beta,\beta^{-1}]\hookrightarrow C^\infty(\Lat),\qquad f\beta^k\mapsto (\varphi^*f)\ell_2^{-k}. \label{eq:coeffinclude}
\eeq
We observe that the image of this map is precisely $\bigoplus_{k\in \Z} C^\infty_k(\Lat)$ for
$$
C^\infty_k(\Lat):=\{f\in C^\infty(\Lat)\mid f(\dil^2\ell_1,\bar\dil^2\ell_1,\dil^2\ell_2\bar\dil^2\ell_2)=\dil^{-k}f(\ell_1,\bar\ell_1,\ell_2,\bar\ell_2)\}
$$
the vector space of smooth functions of weight $k/2$, where $(\dil,\bar\dil)$ are the standard coordinates on $U(1)\simeq \Spin(2)$. Indeed, $C^\infty(\HH\times\R_{>0})$ includes as $C^\infty_0(\Lat)\simeq C^\infty(\Lat)^{\Spin(2)}$, $C^\infty_k(\Lat)=\{0\}$ for $k$ odd, and there are isomorphisms of vector spaces~$C^\infty_{2k}(\Lat)\stackrel{\sim}{\to} C^\infty_0(\Lat)\simeq C^\infty(\Lat)^{\Spin(2)}$ gotten by multiplication with $\ell_2^{k}$.

\begin{lem} 
An element $\omega\in C^\infty(\mathcal{L}^{2|1}_0(M))^{\Spin(2)}\subset C^\infty(\mathcal{L}^{2|1}_0(M))^{\Z/2}$ expressed in the form~\eqref{eq:spininvt} has $\omega_0,\omega_1,\omega_2$ in the image of the inclusion
\beq
&&\Omega^\bullet(M;C^\infty(\HH\times \R_{>0})[\beta,\beta^{-1}])\hookrightarrow \Omega^\bullet(M;C^\infty(\Lat)),\qquad |\beta|=-2\label{eq:coeffinclude2}
\eeq
determined by the map~\eqref{eq:coeffinclude} on coefficients, where $\omega_0$ is in the image of an element of total degree zero and $\omega_1,\omega_2$ are in the image of elements of total degree~$-1$.
\end{lem} 
\bp
From the description of the $\Spin(2)$-action in~\eqref{eq:lem:E21action}, if $\omega\in C^\infty(\mathcal{L}^{2|1}_0(M))^{\Spin(2)}\subset C^\infty(\mathcal{L}^{2|1}_0(M))^{\Z/2}$, we obtain the refinement of the conditions from~\eqref{eq:spininvt},
$$
\omega_0\in \bigoplus_{k\in \Z} \Omega^{2k}(M;C^\infty_{2k}(\Lat))\simeq \bigoplus_{k\in \Z} \Omega^{2k}(M;\ell^{-k}_2C^\infty_0(\Lat))\subset \Omega^\ev(M;C^\infty_0(\Lat)[\ell_2^{-1}]), $$$$ \omega_1,\omega_2\in \bigoplus_{k\in \Z} \Omega^{2k-1}(M;C^\infty_{2k}(\Lat))\simeq \bigoplus_{k\in \Z} \Omega^{2k-1}(M;\ell^{-k}_2C^\infty_0(\Lat))\subset \Omega^\odd(M;C^\infty_0(\Lat)[\ell_2^{-1}]).
$$
This gives the description
\beq
&&\omega=(\vol/\ell_2)^{\deg/2}\omega_0'+2\lambda_1(\vol/\ell_2)^{(\deg+1)/2}\omega_1'+2\lambda_2(\vol/\ell_2)^{(\deg+1)/2}\omega_2',
\eeq
where $\omega_0',\omega_1',\omega_2'\in \Omega^\bullet(M;C^\infty(\Lat)^{\Spin(2)})\simeq \Omega^\bullet(M;C^\infty_0(\Lat))$ are $\Spin(2)$-invariant. After identifying $\ell_2$ with $\beta^{-1}$ as per~\eqref{eq:coeffinclude}, we obtain the claimed description.
\ep

The following allows us to recast the invariance condition as a failure of $Z=\omega_0$ to have holomorphic dependence on the conformal modulus and be independent of volume. 

\begin{lem}\label{lem:Qclosed}
A $\E^{2|1}\rtimes \Spin(2)$-invariant function on $\mathcal{L}_0^{2|1}(M)$ is equivalent to a triple $(Z,Z_{\bar \tau},Z_v)$ where $Z\in \Omega^\bullet(M;C^\infty(\HH\times \R_{>0})[\beta,\beta^{-1}])$ has total degree zero and $Z_v,Z_{\bar\tau}\in \Omega^\bullet(M;C^\infty(\HH\times \R_{>0})[\beta,\beta^{-1}])$ have total degree~$-1$ and satisfy
\beq
&&\dR Z=0,\quad \partial_vZ=\dR Z_v,\quad \partial_{\bar \tau}Z=\dR Z_{\bar\tau}\label{eq:Qclosed21}
\eeq
for coordinates $(\tau,\bar\tau)$ on $\HH$ and $v$ on $\R_{>0}$. \end{lem}
\bp

For the image of $Z$ under~\eqref{eq:coeffinclude2}, we differentiate
\beq
&&\resizebox{.95\textwidth}{!}{$
\partial_{\bar \ell_1}Z(\ell_1/\ell_2,\bar\ell_1/\bar\ell_2,\vol)=\frac{1}{\bar\ell_2}\partial_{\bar \tau}Z-\frac{\ell_2}{2i}\partial_v Z,\ \ \partial_{\bar \ell_2}Z(\ell_1/\ell_2,\bar\ell_1/\bar\ell_2,\vol)=-\frac{\bar\ell_1}{\bar\ell_2^2}\partial_{\bar \tau}Z+\frac{\ell_1}{2i}\partial_v Z.$}\label{Eq:chainrule}
\eeq
The result then follows from comparing with~\eqref{eq:E21}: writing $\partial_{\bar\ell_1}Z$ and $\partial_{\bar\ell_2}Z$ as $\dR$-exact forms is equivalent to writing $\partial_{\bar\tau}Z$ and $\partial_vZ$ as $\dR$-exact forms. 
\ep


\begin{defn} A function $f\in C^\infty(\HH\times \R_{>0})$ \emph{has weight $(k,\bar k)\in \Z\times \Z$} if 
$$
f\left(\frac{a\tau+b}{c\tau+d},v\right)=(c\tau+d)^k(c\bar\tau+d)^{\bar k}f(\tau,v). 
$$
Let $\MF_{k,\bar k}\subset C^\infty(\HH\times\R_{>0})$ denote the $\C$-vector space of functions with weight $(k,\bar k)$. 
\end{defn}

Consider the inclusion 
\beq
\bigoplus_{k\in \Z} \MF_{k,\bar k}\hookrightarrow C^\infty(\HH\times\R_{>0})[\beta,\beta^{-1}]\qquad f\mapsto \beta^kf, \ f\in \MF_{k,\bar k}. \label{eq:anotherinclusion}
\eeq

\begin{lem}\label{lem:SL2} In the notation of Lemma~\ref{lem:Qclosed}, a triple $(Z,Z_v,Z_{\bar\tau})$ determines an $\SL_2(\Z)$-invariant function on $\mathcal{L}_0^{2|1}(M)$ when
$$
Z\in \bigoplus_{k\in \Z} \Omega^{2k}(M;\MF_{k,0}),\quad Z_v\in  \bigoplus_{k\in \Z} \Omega^{2k-1}(M;\MF_{k,0}),\quad Z_{\bar\tau} \in  \bigoplus_{k\in \Z} \Omega^{2k-1}(M;\MF_{k,2})
$$ 
using~\eqref{eq:anotherinclusion} to identify the above with elements of $\Omega^\bullet(M;C^\infty(\HH\times\R_{>0})[\beta,\beta^{-1}])$. 
\end{lem}
\bp
We observe that 
$$
\ell_2\mapsto c\ell_1+d\ell_2=\ell_2(c\tau+d)\qquad \tau=\ell_1/\ell_2, \ \left[\begin{array}{cc} a & b \\ c & d\end{array}\right]\in \SL_2(\Z)
$$
for the $\SL_2(\Z)$-action on $\Lat$, so that~\eqref{eq:coeffinclude} is an $\SL_2(\Z)$-invariant inclusion for the action on $\HH$ by fractional linear transformations and $
\beta\mapsto \beta/(c\tau+d)$. The $\SL_2(\Z)$-invariant property for $Z$ then follows directly. The properties for $Z_v$ and $Z_{\bar\tau}$ can either be deduced from the fact that~\eqref{eq:Qclosed21} are $\SL_2(\Z)$-invariant equations, or by (a direct but tedious computation) using~\eqref{Eq:chainrule} to write $Z_v$ and $Z_{\bar\tau}$ in terms of $\omega_0$ and $\omega_1$, and then applying the $\SL_2(\Z)$-actions on $\omega_0,\omega_1,\omega_2$ computed in Lemma~\ref{lem:E21action}. 
\ep

\begin{proof}[Proof of Proposition~\ref{prop:compute21}]
The result follows from Lemma~\ref{lem:Qclosed} and~\ref{lem:SL2}. 
\ep
\begin{rmk}\label{rmk:STannounce}
As announced in~\cite[Theorem~1.15]{ST11}, a $2|1$-Euclidean field theory over $M=\pt$ has a partition function valued in \emph{integral} modular forms. Theorem~\ref{thm} when $d=2$ specializes to the holomorphy and modularity statements in this result when~$M=\pt$; generalizing the integrality statement would require one to consider the values of field theories on super annuli with maps to~$M$. 
\end{rmk}

\begin{rmk}
The Lie groupoid $\Lat\sq \Spin(2)\times \SL_2(\Z)$ gives a presentation of the moduli stack of Euclidean tori with periodic-periodic spin structure and choice of basepoint, where $\SL_2(\Z)\times \Spin(2)$ acts via the restriction of the action from Lemma~\ref{lem:E21action}. The involution generated by $-1\in U(1)\simeq \Spin(2)$ is the \emph{spin flip automorphism} which acts trivially on the underlying Euclidean torus and by the parity involution on the spinor bundle. Consider the subspace $\HH\times \R_{>0}\subset \Lat$ of based lattices whose 2nd generator $\ell_2\in \R_{>0}\subset \C^\times$ is positive and real. Since every based lattice can be rotated to one of this form (using the action of $\Spin(2)$ on $\Lat$) the full subgroupoid of $\Lat\sq \Spin(2)\times \SL_2(\Z)$ with the objects $\HH\times \R_{>0}\subset \Lat$ is equivalent to $\Lat\sq \Spin(2)\times \SL_2(\Z)$. Since $\{\pm 1\}\subset \Spin(2)$ acts trivially on the subspace $\HH\times \R_{>0}\subset \Lat$, the manifold of morphisms in this full subgroupoid is $\HH\times \R_{>0}\times \{\pm 1\}\times \SL_2(\Z)$. Composition of morphisms gives the set $\{\pm 1\}\times \SL_2(\Z)$ the structure of a group which turns out to be the metaplectic double cover ${\rm MP}_2(\Z)$ of $\SL_2(\Z)$. There is a functor between Lie groupoids $u\colon \HH\times \R_{>0}\sq {\rm MP}_2(\Z)\to \HH\sq {\rm MP}_2(\Z)$, where the target is a standard presentation for the stack of complex analytic elliptic curves endowed with a periodic-periodic spin structure. Geometrically, the functor $u$ extracts the underlying complex analytic elliptic curve with spin structure. 

Finally, observe there is a functor $\Lat\sq \Spin(2)\times \SL_2(\Z)\to \mathcal{M}^{2|1}\sq \Euc_{2|1}$, so a family of Euclidean tori with spin structure and choice of basepoint determines a family of super tori. Our arguments involving super tori do not encounter the metaplectic double cover because at the outset (in Lemma~\ref{lem:spininvt}) we restrict to functions invariant under the spin flip automorphism. Hence only the quotient ${\rm MP}_2(\Z)/\{\pm 1\}\simeq \SL_2(\Z)$ features in our arguments. 
\end{rmk}

\subsection{Weak modular forms and complexified TMF} \label{sec:MF}

\begin{defn}\label{defn:mf} \emph{Weak modular forms of weight $k$} are holomorphic functions $f\in \mathcal{O}(\HH)$ satisfying
$$
f\left(\frac{a\tau+b}{c\tau+d}\right)=(c\tau+d)^kf(\tau)\qquad \tau\in \HH, \ \left[\begin{array}{cc} a & b \\ c & d\end{array}\right]\in \SL_2(\Z). 
$$
Let $\MF_k$ denote the $\C$-vector space of weak modular forms of weight $k$. Define the graded ring of weak modular forms~$\MF$ as the graded vector space
$$
\MF=\bigoplus_{k\in \Z} \MF^k\qquad \MF^k:=\left\{\begin{array}{ll} \MF_{k/2} & k\ {\rm even} \\ 0 & k\ {\rm odd}\end{array}\right.
$$
with ring structure from multiplication of functions on $\HH$. 
\end{defn}


Cohomology with coefficients in weak modular forms is the object that naturally appears when studying derived global sections of the elliptic cohomology sheaf in the complex analytic context. Indeed, complex analytic elliptic cohomology assigns to a smooth manifold~$M$ a sheaf $\Ell(M)$ of differential graded algebras on the orbifold $\HH\sq \SL_2(\Z)$ with values
\beq
\Ell(M)(U):=(\mathcal{O}(U;\Omega^\bullet(M)[\beta,\beta^{-1}]),\dR)\qquad {\rm for} \qquad U\subset \HH.\label{eq:Ellsheaf}
\eeq
The $\SL_2(\Z)$-equivariance data for this sheaf comes from pulling back functions along fractional linear transformations and sending $\beta\mapsto (c\tau+d)\beta$. 
This connects with standard definitions of elliptic cohomology in homotopy theory (e.g.,~\cite[Definition~1.2]{Lurie_Elliptic}) by identifying $\HH\sq \SL_2(\Z)$ with the moduli stack of complex analytic elliptic curves, and values~\eqref{eq:Ellsheaf} with the de~Rham complex for 2-periodic cohomology with coefficients in $\mathcal{O}(U)$. Using the Dolbeault resolution of holomorphic functions on~$\HH$, the complex $(\Omega^\bullet(M;\Omega^{0,*}(\HH)[\beta,\beta^{-1}])^{\SL_2(\Z)},\dR+\bar\partial)$ computes the derived global sections (i.e., the hypercohomology) of the elliptic cohomology sheaf $\Ell(M)$. Since $\HH$ is Stein, the inclusion 
$$
\mathcal{O}(\HH)\hookrightarrow (\Omega^{0,*}(\HH),\bar\partial)
$$
is a quasi-isomorphism. Hence, derived global sections of the elliptic cohomology sheaf are cohomology with values in weak modular forms, 
$$
\H(M;\mathcal{O}(\HH)[\beta,\beta^{-1}])^{\SL_2(\Z)}\simeq \H(M;\MF).
$$
We refer to~\cite[\S3]{BElocalization} for details. 

 A weak modular form is a \emph{weakly holomorphic modular form} if it is meromorphic as $\tau\to i\infty$. For $M$ compact, cohomology with values in weakly holomorphic modular forms is isomorphic to the complexification of topological modular forms,
\beq
\TMF(M)\otimes \C&\simeq& \H(M;\TMF(\pt)\otimes \C)\subset \H(M;\MF)\label{eq:TMFchern}\\
\TMF(\pt)\otimes \C&\simeq& \{{\rm weakly \ holomorphic\ modular\ forms}\}\subset \MF\nonumber 
\eeq
and the inclusion on the right regards a weakly holomorphic modular form as a weak modular form. We expect the image of $2|1$-Euclidean field theories along~\eqref{eq:thm2} to satisfy this meromorphicity property at~$i\infty$, and hence have image in the subring $\TMF(M)\otimes \C$. This follows from an ``energy bounded below" condition discussed for $M=\pt$ in~\cite[\S3]{ST11}. However, proving that the image of field theories satisfies this condition requires that one analyze the values of field theories on super tori \emph{and} super annuli. 

\subsection{Concordance classes of functions} \label{sec:d2cocycle}

The cocycle map~\eqref{eq:thm2} can be factored through a complex that computes the derived global sections of the elliptic cohomology sheaf, namely the complex $(\Omega^\bullet(M;\Omega^{0,*}(\HH)[\beta,\beta^{-1}])^{\SL_2(\Z)},\dR+\bar\partial)$ described above.

\begin{defn}\label{defn:cocycle2}
Using the notation from Proposition~\ref{prop:compute21}, for each $\mu\in \R_{>0}$, define a map
\beq
\widehat{\rm cocycle}_\mu\colon C^\infty(\mathcal{L}^{2|1}_0(M))^{\Euc_{2|1}}&\to&  {\rm Z}^0(\Omega^\bullet(M;\Omega^{0,*}(\HH)[\beta,\beta^{-1}]),\dR+\bar\partial)^{\SL_2(\Z)}\nonumber\\
(Z,Z_{\bar\tau},Z_v)&\mapsto &Z(\mu)+d\bar\tau Z_{\bar \tau}(\mu)\nonumber
\eeq
where the evaluation is at tori with volume $v=\mu\in \R_{>0}$. 
\end{defn}

\begin{lem} \label{lem:cohomologicalmap}
The composition
$$
C^\infty(\mathcal{L}^{2|1}_0(M))^{\Euc_{2|1}}\stackrel{\widehat{\rm cocycle}_\mu}{\longrightarrow}   {\rm Z}^0(\Omega^\bullet(M;\Omega^{0,*}(\HH)[\beta,\beta^{-1}]),\dR+\bar\partial)^{\SL_2(\Z)}\stackrel{{\rm de\ Rham}}{\longrightarrow} \H(M;\MF)
$$
is independent of $\mu$ and agrees with~\eqref{eq:thm2}. 
\end{lem}

\bp
First we verify that the map in Definition~\ref{defn:cocycle2} is well-defined. By Proposition~\ref{prop:compute21}, the image is contained in the subspace of degree zero cocycles:
$$
(\dR+\bar\partial)(Z(\mu)+d\bar\tau Z_{\bar \tau}(\mu))=d\bar\tau \partial_{\bar \tau} Z(\mu)-d\bar\tau \dR Z_{\bar \tau}(\mu)=0.
$$
The image is $\SL_2(\Z)$-invariant by Lemma~\ref{lem:SL2}. The remainder of the proof is completely analogous to that of Lemma~\ref{lem:independentofmu}. 
\ep

%

\begin{proof}[Proof of Proposition~\ref{mainprop}, $d=2$] 
Proposition~\ref{prop:compute21} implies that $M\mapsto C^\infty(\mathcal{L}^{2|1}_0(M))^{\Euc_{2|1}}$ is a sheaf on the site of smooth manifolds. The map in Definition~\ref{defn:cocycle2} is a morphism of sheaves
$$
\widehat{\rm cocycle}_\mu\colon C^\infty(\mathcal{L}^{2|1}_0(-))^{\Euc_{2|1}}\to {\rm Z}^0(\Omega^\bullet(-;\Omega^{0,*}(\HH)[\beta,\beta^{-1}]),\dR+\bar\partial)^{\SL_2(\Z)}.
$$
When evaluated on a manifold $M$, concordance classes of sections of the target are cohomology classes. This completes the proof. 
\end{proof}

%
%


\subsection{The elliptic Euler class as a cocycle}\label{sec:Witten}

For a real oriented vector bundle $V\to M$ consider the characteristic class
$$
[\Eu(V)]:=\left[\Pf(-\beta F)\exp\left(\sum_{k\ge 1}\frac{\beta^k E_{2k}}{2k(2\pi i)^{2k}} {\rm Tr}(F^{2k})\right)\right]\in \H^{{\rm dim}(V)}(M;C^\infty(\HH)[\beta,\beta^{-1}])^{\SL_2(\Z)}
$$
where $F=\nabla\circ\nabla\in \Omega^2(M;\End(V))$ is the curvature for a choice of a metric compatible connection $\nabla$ on~$V$ and $\Pf(-\beta R)$ is the Pfaffian. The functions $E_{2k}\in C^\infty(\HH)$ are the $2k$th Eisenstein series, where we take $E_2$ to be the modular, nonholomorphic version of the 2nd Eisenstein series,
$$
E_2(\tau,\bar\tau)=\lim_{\epsilon\to 0^+} \sum_{(n,m)\in \Z^2_*} \frac{1}{(n\tau+m)^2|n\tau+m|^{2\epsilon}},\qquad E_2(\tau,\bar\tau)=E_2^{\rm hol}(\tau)-\frac{2\pi i}{\tau-\bar\tau},
$$
whose relationship with the holomorphic (but not modular) 2nd Eisenstein series $E_2^{\rm hol}(\tau)$ is as indicated. For $k>1$, the Eisenstein series $E_{2k}\in \mathcal{O}(\HH)$ are holomorphic. Thus, if 
$$
[p_1(V)]=[{\rm Tr}(F^{2})/(2(2\pi i)^{2})]\in \H^4(M;\R)
$$
vanishes, then $[\Eu(V)]\in \H^{{\rm dim}(V)}(M;\mathcal{O}(\HH)[\beta,\beta^{-1}])^{\SL_2(\Z)}$ is a holomorphic class. 

When ${\rm dim}(V)=24k$, we may ask for a preimage of $\Delta^k[\Eu(V)]\in \H^{0}(M;\MF)$ under the cocycle map~\eqref{eq:thm2}, where $\Delta$ is the modular discriminant. We start with the differential form refinement of~$\Eu(V)$, evident from its definition above
$$
\Eu(V)\in\Omega^\bullet(M;C^\infty(\HH)[\beta,\beta^{-1}]),\quad \partial_{\bar \tau}\Eu(V)=\frac{\beta^2{\rm Tr}(F^2)}{4\pi i (\tau-\bar\tau)^2}\Eu(V),
$$
and whose failure to be holomorphic is as indicated. Since $\partial_v\Eu(V)=0$, we may choose $Z=\Delta^k\Eu(V)$ and $Z_v=0$. The remaining data to promote $\Delta^k\Eu(V)$ to a function on $\mathcal{L}^{2|1}_0(M)$ is a choice of coboundary $\partial_{\bar \tau}(\Delta^k\Eu(V))=\dR Z_{\bar\tau}$, which in turn is determined by~$H\in \Omega^3(M)$ with $\dR H=p_1(V)$, i.e., a rational string structure. This identifies the set of rational string structures on $(V,\nabla)$ with choices of lift of $\Delta^k[\Eu(V)]$ to a function on~$\mathcal{L}^{2|1}_0(M)$. We expect a similar story without the dimension restriction on~$V$ and the factors of~$\Delta$ though an enhancement of~\eqref{eq:thm2} that incorporates a degree~$n$ twist~\cite[\S5]{ST11}. 


\bibliographystyle{amsalpha}
\bibliography{references}

\end{document}